\documentclass[12pt,reqno]{siugrad51new}
\usepackage{amsmath, amssymb, amsthm, latexsym, mathtools,bold-extra}
\makeatletter
\providecommand*{\dashv}{%
  \mathrel{%
    \mathpalette\@dashv\vdash
  }%
}
\newcommand*{\@dashv}[2]{%
  \reflectbox{$\m@th#1#2$}%
}
\usepackage{tikz-cd}
\usepackage[colorlinks=false,hyperfootnotes=false]{hyperref}
\renewcommand*\thesection{\arabic{section}}
\theoremstyle{plain}
\newtheorem{theorem}{Theorem}[section]
\newtheorem{corollary}[theorem]{Corollary}
\newtheorem{lemma}[theorem]{Lemma}
\newtheorem{proposition}[theorem]{Proposition}
\theoremstyle{definition}

\newtheorem{fact}[theorem]{Fact}
\newtheorem{definition}[theorem]{Definition}
\newtheorem{remark}[theorem]{Remark}
\newtheorem{assum}[theorem]{Assumption}
\newtheorem{ques}[theorem]{Question}
\newcommand{\han}{\beth_{(2^{\ls})^+}}
\newcommand{\lk}{\leq_{\bf K}}

\newcommand{\mn}{\mathfrak{C}}
\newcommand{\oop}{\operatorname}
\newcommand{\gtp}{\oop{gtp}}
\newcommand{\gs}{\oop{gS}}
\newcommand{\cf}{\oop{cf}}

\newcommand{\ran}{\oop{ran}}

\newcommand{\ls}{\oop{LS}({\bf K})}
\newcommand{\lss}{\oop{LS}({\bf K_1})}
\newcommand{\llk}{\oop{L}({\bf K})}
\newcommand{\defeq}{\vcentcolon=}
\newcommand{\eqdef}{=\vcentcolon}
\def\fork{\mathrel{\raise0.2ex\hbox{\ooalign{\hidewidth$\vert$\hidewidth\cr\raise-0.9ex\hbox{$\smile$}}}}}

\newcommand{\nr}[1]{\lVert #1 \rVert}
\newcommand{\card}[1]{\lvert #1 \rvert}
\newcommand{\al}{{\aleph_0}}

\DeclareMathOperator{\tp}{tp}

\makeatletter
\@ifpackageloaded{hyperref}%
  {\newcommand{\mylabel}[2]
    {\protected@write\@auxout{}{\string\newlabel{#1}{{#2}{\thepage}%
      {\@currentlabelname}{\@currentHref}{}}}}}%
  {\newcommand{\mylabel}[2]
    {\protected@write\@auxout{}{\string\newlabel{#1}{{#2}{\thepage}}}}}
\makeatother
\begin{document}

\pagenumbering{roman}
\setcounter{page}{0}
\newpage
\pagenumbering{arabic}
\setcounter{page}{1}
\parindent=.35in
\begin{center}
        \begin{center}%
         {\Large\bfseries\scshape Hanf number of the\\First stability cardinal in AECs}\\\vspace{1em}{\scshape Samson Leung}\\      
        \end{center}%
\end{center}
{\let\thefootnote\relax\footnote{Date: \today\\
AMS 2020 Subject Classification: Primary 03C48. Secondary: 03C45, 03C55.
Key words and phrases. Abstract Elementary Classes; Stability; First stability cardinal; Tameness; Hanf number.
}} 
\begin{abstract}
We show that $\beth_{(2^{\ls})^+}$ is the lower bound to the Hanf numbers for the length of the order property and for stability in stable abstract elementary classes (AECs). Our examples satisfy the joint embedding property, no maximal model, $(<\al)$-tameness but not necessarily the amalgamation property. We also define variations on the order and syntactic order properties by allowing the index set to be linearly ordered rather than well-ordered. Combining with Shelah's stability theorem, we deduce that our examples can have the order property up to any $\mu<\beth_{(2^{\ls})^+}$. Boney conjectured that the need for joint embedding property for two type-counting lemmas is necessary. We solved the conjecture by showing it is independent of ZFC. Using Galois Morleyization, we give syntactic proofs to known stability results assuming a monster model.
\end{abstract}
\vspace{1em}
\begin{center}
{\bfseries TABLE OF CONTENTS}
\end{center}
\tableofcontents

\section{Introduction}
Semantic order properties (\ref{fsdef}) in abstract elementary classes (AECs) are defined in terms of (semantic) Galois types instead of formulas. They are analogs to syntactic order properties in first-order and infinitary logics. 
In \cite{sh16}, Shelah showed that in $L_{\lambda^+,\omega}$ the (syntactic) order property of length $\beth_{(2^{\lambda})^+}$ implies the order property of arbitrary length. In \cite{sh222}, Grossberg and Shelah introduced the Hanf number of the order property of $L_{\lambda^+,\omega}$ and later \cite[Theorem 2.8]{sh259} gave a lower bound as $\beth_{\lambda^+}$. These bound the Hanf number of order property between $\beth_{\lambda^+}$ and $\beth_{(2^{\lambda})^+}$. However, the example for the lower bound does not readily generalize to (semantic) order properties of AECs. Shelah \cite[Claim 4.6]{sh394} hinted that the upper bound of the order property in AECs is $\beth_{(2^{\ls})^+}$ but it was not known whether it is tight. We present examples (\ref{fsprom}) that $\beth_{(2^{\ls})^+}$ is exact. Our examples satisfy the joint embedding property, no maximal model and $(<\al)$-tameness but the amalgamation property fails. It is open whether the bound can be lowered when one assumes the amalgamation property.

Vasey \cite{s5} extended Shelah \cite{sh3}, Grossberg and Lessmann's \cite{GLspec} results to AECs and showed that assuming the amalgamation property and tameness, the first stability cardinal is bounded above by $\beth_{(2^{\ls})^+}$. It is open whether this bound can be lowered under the amalgamation property. Our examples, which do not satisfy the amalgamation property, show that the lower bound in general is at least $\beth_{(2^{\ls})^+}$. From instability, we can apply Vasey's techniques (which are based on \cite[V.A.]{sh300a}) to derive the order property. This provides an alternative way other than finding the witnesses directly. It is open whether the amalgamation property can lower the bound for the first stability cardinal.

Vasey's result above relies on one direction of \cite[Theorem 3.1]{bon3.1}, which does not use the joint embedding property. The other direction involves lemmas that assume the joint embedding property, which Boney suspected to be necessary. As a side product of our construction, we show in \ref{fscor3} that the need for the joint embedding property is independent of ZFC; and we find an example and a counterexample under different set theoretic assumptions. 

In Section 2, we state our notations and definitions. We also give a shorter proof to Boney's result to motivate \ref{fscor3}. In Section 3, we review results concerning stability and the order property. We give more details for the proof of \cite[Claim 4.6]{sh394}. In Section 4, we construct our main examples in \ref{fsprop} which set a lower bound to stability and a variation of the order property for stable AECs. The variation of the order property is slightly more general by allowing the index set to be linear ordered rather than well-ordered. We will also show \ref{fscor3} as a side product of our construction. In Section 5, we apply the same variation to the syntactic order property which can be combined with Galois Morleyization. We give analogs to Vasey's results with our variation on the order property. In Section 6, we write down the details of Vasey's observation \cite[Fact 4.10]{s5} that Shelah's results in \cite[V.A.]{sh300a} can be applied to AECs under Galois Morleyization. It allows us to deduce (\ref{fsprom}) the order property up to $\beth_{(2^{\ls})^+}$ in our examples in \ref{fsprop}, without finding explicit witnesses. We also apply such technique to bound the first stability cardinal under extra hypotheses. In Section 7, we use Galois Morleyization to recover common stability results where types can be over sets under $\ls$. Vasey in \cite[Section 5]{s5} has done similarly for coheir while we will work on splitting instead. In particular we prove \ref{gvstab} syntactically which is needed for Vasey's upper bound to the first stability cardinal, under the amalgamation property and tameness.

This paper was written while the author was working on a Ph.D. under the direction of Rami Grossberg at Carnegie Mellon University and we would like to thank Prof. Grossberg for his guidance and assistance in my research in general and in this work in particular. We also thank John Baldwin, Hanif Cheung, Marcos Mazari-Armida and Wentao Yang for useful comments.

\section{Preliminaries}

We assume some familiarity with AECs, for example \cite[Chapter 4]{bal}. We will use $\kappa,\lambda,\mu,\chi$ to denote cardinals, $\alpha,\beta,\gamma$ to denote ordinals, $n$ for natural numbers. We define $\kappa^-$ to be the predecessor cardinal (if it exists) or $\kappa$ itself. When we write $\alpha-n$, we assume $\alpha=\beta+n$ for some ordinal $\beta$.

Let ${\bf K}=\langle K,\lk\rangle$ be an AEC. If the context is clear, we write $\leq$ in place of $\lk$. We abbreviate $AP$ the \emph{amalgamation property}, $JEP$ the \emph{joint embedding property} and $NMM$ \emph{no maximal model}. For $M\in K$, write $|M|$ the universe of $M$ and $\nr{M}$ the cardinality of $M$. For Galois types (orbital types) of length $(<\alpha)$, we denote them as $\gs^{<\alpha}(\cdot)$ where the argument can be a set $A$ in some model $M\in K$. In general $\gs^{<\alpha}(A)\defeq\bigcup\{\gs^{<\alpha}(A;M):M\in K, |M|\supseteq A\}$ (under $AP$, the choice of $M$ does not matter). ${\bf K}$ is \emph{$(<\alpha)$-stable in $\lambda$} if for any set $A$ in some model $M\in K$, $\card{A}\leq\lambda$, then $\card{\gs^{<\alpha}(A;M)}\leq\lambda$. We omit ``$(<\alpha)$'' if $\alpha=2$, while we omit ``in $\lambda$'' if there exists such a $\lambda\geq\ls$. Similarly ${\bf K}$ is \emph{$\alpha$-stable in $\lambda$} if for any set $A$ in some model $M\in K$, $\card{\gs^\alpha(A)}\leq\lambda$. Tameness will be defined in \ref{tame2}. We allow stability and tameness under $\ls$, especially in Section \ref{synspl}.

Given a $\forall\exists$ theory $T$ and a set of $L(T)$-types $\Gamma$, $EC(T,\Gamma)$ is the class of models of $T$ such that $\Gamma$ is not realized by any elements. If we order $EC(T,\Gamma)$ by $L$-substructures, it forms an AEC with $\ls=|L(T)|$. $\delta(\lambda,\kappa)$ is the least ordinal $\delta$ such that: for any $T,\Gamma$ with $|L(T)|\leq\lambda$, $|\Gamma|\leq\kappa$, $\{P,<\}\subseteq L(T)$ where $P$ is a unary predicate, $<$ is a linear order on $P$, if there is a model $M\in EC(T,\Gamma)$ whose $(P^M,<^M)$ has order type $\geq\delta$, then there is a model $N\in EC(T,\Gamma)$ whose $(P,<)$ is not well-ordered.

Recall the classical theorem:  notice in the proof that $X$ can be a linear order while $Y'$ can be its suborder.
\begin{theorem}[Erd\H{o}s-Rado Theorem]\mylabel{erthm}{Theorem \thetheorem}
Let $\lambda$ be an infinite cardinal. For $n<\omega$,
$$\beth_n(\lambda)^+\rightarrow(\lambda^+)^{n+1}_\lambda$$
In other words, for any $|X|\geq\beth_n(\lambda)^+$, any $f:[X]^{n+1}\rightarrow\lambda$, there is $X'\subseteq X$ such that $|X'|\geq\lambda^+$ and $f\restriction[X']^{n+1}$ is constant.
\end{theorem}

\begin{proof}
We adapt the proof in \cite[Theorem 5.1.4]{marker} because it does not require the set $X$ to be well-ordered. We prove by induction:
When $n=0$, the statement is $\lambda^+\rightarrow(\lambda^+)^1_\lambda$. Let $X$ be of size $\geq\lambda^+$. We need to color its elements with $\lambda$-many colors. By pigeonhole principle, it is possible to find $X'\subseteq X$ of size $\geq\lambda^+$ such that $f\restriction X'$ is constant.

Assume the statement is true for $n-1$. We need to show $\beth_n(\lambda)^+\rightarrow(\lambda^+)^{n+1}_\lambda$. Let $X$ be of size $\beth_n(\lambda)^+$, $f:[X]^{n+1}\rightarrow\lambda$. For $x\in X$, define $f_x:[X-\{x\}]^n\rightarrow\lambda$ by $f_x(Y)\defeq f(Y\cup\{x\})$. We build $\langle X_\alpha:\alpha<\beth_{n-1}(\lambda)^+\rangle$ increasing and continuous subsets of $X$ such that for $\alpha<\beth_n(\lambda)^+$, $|X_\alpha|=\beth_n(\lambda)$. For the base step, take any $X_0\subseteq X$ of size $\beth_n(\lambda)$. Suppose $X_\alpha$ is constructed, we build $X_{\alpha+1}$ satisfying:

\begin{enumerate}
\item $X_\alpha\subseteq X_{\alpha+1}\subseteq X$
\item $|X_{\alpha+1}|=\beth_n(\lambda)$
\item For any $Y\subseteq X_\alpha$ of size $\beth_{n-1}(\lambda)$, any $x\in X-Y$, there is $x'\in X_{\alpha+1}-Y$ such that $f_x\restriction[Y]^n=f_{x'}\restriction[Y]^n$.
\end{enumerate}

The above is possible by a counting argument: the number of possible $Y$ is $$|X_\alpha|^{\beth_{n-1}(\lambda)}=\beth_{n}(\lambda)^{\beth_{n-1}(\lambda)}=(2^{\beth_{n-1}(\lambda)})^{\beth_{n-1}(\lambda)}=2^{\beth_{n-1}(\lambda)}=\beth_{n}(\lambda).$$ Given $Y$, the number of possible $h:[Y]^n\rightarrow\lambda$ is bounded by 
$$\lambda^{\beth_{n-1}(\lambda)}=2^{\beth_{n-1}(\lambda)}=\beth_{n}(\lambda)$$ Therefore, it suffices to add $\beth_{n}(\lambda)\cdot\beth_{n}(\lambda)=\beth_{n}(\lambda)$-many witnesses from $X$ to $X_\alpha$. Define $X'=\bigcup\{X_\alpha:\alpha<\beth_{n-1}(\lambda)^+\}$. Notice that $|X'|=\beth_{n}(\lambda)<|X|$. For any $Y\subseteq X'$ of size $\beth_{n-1}(\lambda)$, by a cofinality argument $Y\subseteq X_\alpha$ for some $\alpha<\beth_{n-1}(\lambda)^+$. So for any $x\in X-Y$, there is $x'\in X_{\alpha+1}-Y\subseteq X'-Y$ such that $f_x\restriction[Y]^n=f_{x'}\restriction[Y]^n$.

Pick any $x\in X-X'$ and build $Y=\{y_\alpha:\alpha<\beth_{n-1}(\lambda)^+\}\subseteq X'$ such that $f_{y_\alpha}\restriction[\{y_\beta:\beta<\alpha\}]^n=f_{x}\restriction[\{y_\beta:\beta<\alpha\}]^n$ ($y_0\in X'$ can be any element). By inductive hypothesis on $Y$ and $f_x$, we can find $Y'\subseteq Y$ of size $\geq\lambda^+$ such that $f_x\restriction[Y']^n$ is constant. We check that $Y'$ is as desired: let $A\in[Y']^{n+1}$ and write $A=\{y_{\alpha_1},\dots,y_{\alpha_{n+1}}\}$ where $\alpha_1<\cdots<\alpha_{n+1}<\beth_{n-1}(\lambda)^+$. $$f(A)=f_{y_{\alpha_{n+1}}}(A-\{y_{\alpha_{n+1}}\})=f_x(A-\{y_{\alpha_{n+1}}\})$$ which is constant because $f_x$ is constant on $[Y']^n\ni A-\{y_{\alpha_{n+1}}\}$. 
\end{proof}
The following \ref{bonthm} and \ref{gvstab} are only used in the proof of \ref{fscor2}(1). We will streamline Boney's proof of \ref{bonlem4} by omitting the ambient models (otherwise it would involve a lot of bookkeeping and direct limits). We will clarify the relationship between \ref{bonlem4} and \ref{bonthm} and show that $JEP$ in \ref{bonthm} is not needed. If we work in a monster model $\mn$, we can also allow stability over sets (of size $<\ls$), but we keep the original formulation to state \ref{bonresrmk} more clearly.

\begin{theorem}\mylabel{bonthm}{Theorem \thetheorem}
Let ${\bf K}$ be an AEC and $\lambda\geq\ls$. Suppose ${\bf K}$ has $\lambda$-$AP$ and is stable in $\lambda$. For any ordinal $\alpha\geq1$ with $\lambda^{|\alpha|}=\lambda$, ${\bf K}$ is $\alpha$-stable in $\lambda$.
\end{theorem}

The requirement $\lambda^{|\alpha|}=\lambda$ cannot be improved: let $\lambda^{|\alpha|}>\lambda$, take ${\bf K}$ be the well-orderings of type at most $\lambda$ and $\lk$ by initial segments. Then it is stable in $\lambda$ because there are only $\lambda$-many elements in the unique maximal model (which witnesses $AP$). It is not $\alpha$-stable because each element has different types, so the number of $\alpha$-types is exactly $\lambda^{|\alpha|}>\lambda$.

We will prove the theorem through a series of lemmas. We may assume $\alpha$ to be a cardinal $\kappa=|\alpha|$. Denote $\gs^1_\lambda\defeq\sup\{|\gs^1(M)|:M\in K,\nr{M}=\lambda\}$ and similarly $\gs^\kappa_\lambda\defeq\sup\{|\gs^\kappa(M)|:M\in K,\nr{M}=\lambda\}$.
\begin{lemma}\mylabel{bonlem1}{Lemma \thetheorem}
Suppose $\kappa\geq\lambda$, then $(\gs^1_\lambda)^\kappa=\gs^\kappa_\lambda=2^\kappa$.
\end{lemma}
\begin{proof}
$2^\kappa\leq(\gs_\lambda^1)^\kappa\leq(2^\lambda)^\kappa=2^\kappa$. Pick any $\nr{M}=\lambda$ and two distinct elements $\{a,b\}$ from $|M|$. Form binary sequences from $\{a,b\}$ of length $\kappa$, which shows $2^\kappa\leq\gs_\lambda^\kappa\leq2^{\lambda+\kappa}=2^\kappa$. 
\end{proof}
\begin{lemma}\emph{\cite[Proposition 2.7]{bon3.1}}\mylabel{bonlem2.5}{Lemma \thetheorem}
Suppose $\kappa\leq\lambda$. If in addition ${\bf K}$ has $\lambda$-$JEP$ and $\cf(\gs_\lambda^\kappa)\leq\lambda$, then there is $M\in K$, $\nr{M}=\lambda$ such that $|\gs^\kappa(M)|=\gs_\lambda^\kappa$.
\end{lemma}
\begin{proof}
Pick $\langle M_i:i<\mu\rangle$ (not necessarily increasing) witnessing $\mu\defeq\cf(\gs_\lambda^\kappa)\leq\lambda$. By $\lambda$-$AP$ and $\lambda$-$JEP$, obtain $M$ of size $\lambda$ such that $M\geq M_i$ for all $M_i$. $|\gs^\kappa(M)|\geq\sup_{i<\mu}|\gs^\kappa(M_i)|=\gs_\lambda^\kappa$. 
\end{proof}
\begin{lemma}\emph{\cite[Theorem 3.2]{bon3.1}}\mylabel{bonlem2}{Lemma \thetheorem}
Suppose $\kappa\leq\lambda$. If in addition ${\bf K}$ has $\lambda$-$JEP$, then $(\gs^1_\lambda)^\kappa\leq\gs^\kappa_\lambda$.
\end{lemma}
\begin{proof}
Given $M\in K$ of size $\lambda$. We show that $|\gs^\kappa(M)|\geq|\gs^1(M)|^\kappa$, which does not use $\lambda$-$JEP$. By $\lambda$-$AP$, pick $N\geq M$ (perhaps of size greater than $\lambda$) such that $N$ realizes $\gs(M)$, say by $\langle a_i:i<|\gs(M)|\rangle$. Form sequences of length $\kappa$ from the $a_i$, there are $|\gs(M)|^\kappa$-many sequences. They realize distinct types in $\gs^\kappa(M)$ by checking each coordinate. 

Suppose $\cf(\gs_\lambda^1)\leq\kappa$, then $\cf(\gs_\lambda^1)\leq\lambda$. Substitute $\kappa=1$ in \ref{bonlem2.5} (which uses $\lambda$-$JEP$), there is $M^*\in K_\lambda$ such that $|\gs^1(M^*)|=\gs_\lambda^1$. Thus $(\gs^1_\lambda)^\kappa=|\gs^1(M^*)|^\kappa\leq|\gs^\kappa(M^*)|\leq\gs^\kappa_\lambda$.

Now suppose $\cf(\gs_\lambda^1)>\kappa$, then a cofinality argument shows the second equality below: \begin{align*}(\gs_\lambda^1)^\kappa\defeq&(\sup\{|\gs^1(M)|:M\in K, \nr{M}=\lambda\})^\kappa\\=&\sup\{|\gs^1(M)|^\kappa:M\in K, \nr{M}=\lambda\}\\\leq&\sup\{|\gs^\kappa(M)|:M\in K, \nr{M}=\lambda\}\eqdef\gs_\lambda^\kappa\end{align*}
\end{proof}
\begin{remark}\mylabel{bonresrmk}{Remark \thetheorem}
After \cite[Proposition 2.7]{bon3.1}, Boney suggested that $\lambda$-$JEP$ might be necessary. We will show in \ref{fscor3} that the need of $\lambda$-$JEP$ is independent of ZFC for the above two lemmas.
\end{remark}
\begin{ques}
In \cite[Proposition 2.7]{bon3.1}, there is an alternative hypothesis to \ref{bonlem2.5} where $\cf(\gs_\lambda^\kappa)\leq\lambda$ is replaced by a stronger assumption $I(K,\lambda)\leq\lambda$. Would $\lambda$-$JEP$ be necessary in this case or is it again independent of ZFC? An answer would shed light on the relationship between stability and the number of nonisomorphic models.
\end{ques}
\begin{lemma}\emph{\cite[Theorem 3.5]{bon3.1}}\mylabel{bonlem3}{Lemma \thetheorem}
Suppose $\kappa\leq\lambda$, then $(\gs^1_\lambda)^\kappa\geq\gs^\kappa_\lambda$.
\end{lemma}
\begin{proof}
First we describe the proof strategy: for a fixed model $M$, we show that $\gs^\kappa(M)$ is bounded above by $(\gs^1_\lambda)^\kappa$. To do so, we build a $\gs^1_\lambda$-branching tree of models of height $\kappa$ and list the possible 1-types of each model. For each $\kappa$-type in $\gs^\kappa(M)$, we map it injectively to a branch of the tree (which is a sequence in $(\gs^1_\lambda)^\kappa$), according to the 1-types of the elements from that sequence.

Let $\mu\defeq\gs_\lambda^1$. Fix an arbitrary $M\in K$ with $\nr{M}=\lambda$. Write $\gs^\kappa(M)=\langle p_k:k<\chi\rangle$ where $p_k$ are distinct. Fix $\bar{a}_k\defeq\{a_k^\alpha:\alpha<\kappa\}\vDash p_k$. Construct a tree of models $\langle M_\nu\in K_\lambda:\nu\in\mu^{<\kappa} \rangle$ as follows: $M_{\langle\rangle}\defeq M$, take union at limit stages. Suppose $M_\nu$ is built for some $\nu\in\mu^{<\kappa}$. Enumerate without repetition $\gs^1(M_\nu)=\langle q_i^\nu:i<\chi_\nu\rangle$ for some $\chi_\nu\leq\mu$. For $i<\chi_\nu$, define $M_{\nu^\frown i}\in K_\lambda$ with $M_{\nu^\frown i}\geq M_\nu$ and containing some $c_i^\nu\vDash q_i^\nu$. For $\chi_\nu\leq i<\mu$ (if there is any), give a default value to $M_{\nu^\frown i}\defeq M_\nu$. 
Now we map each $p_k\in\gs^\kappa(M)$ to $\eta_k\in\mu^\kappa$ as follows: suppose $\nu\defeq\eta_k\restriction\alpha$ has been defined for some $\alpha<\kappa$, we set $\eta_k[\alpha]$ to be the minimum $i<\chi_\nu$ (which is the same as requiring $i<\mu$) such that $a_k^\alpha$ realizes $q_i^\nu$. In other words, we decide the $\alpha$-th element of the branch based on the type of the $\alpha$-th element of $\bar{a}_k$ over the current node (model). 

It remains to check that the map is injective. Let $k<\chi$, we build $\langle f_\alpha:\alpha\leq\kappa \rangle$ increasing and continuous such that $f_{\alpha}$ maps $a_k^\beta$ to $c_{\eta_k[\beta]}^{\eta_k\restriction\beta}$ for all $\beta\leq\alpha$ while fixing $|M|$. Take $f_{-1}\defeq\oop{id}_M$ and we handle the successor case: suppose $f_\alpha$ has been constructed. There is $g:a_k^\alpha\mapsto c_{\eta_k[\alpha]}^{\eta_k\restriction\alpha}$ fixing $|M_{\eta_k\restriction\alpha}|\supseteq\big\{c_{\eta_k[\beta]}^{\eta_k\restriction\beta}:\beta<\alpha\big\}$ by type equality. Let $f_{\alpha+1}\defeq g\circ f_\alpha$. Now $f_\kappa$ witnesses that $\bar{a}_k$ and $\langle c_{\eta_k[\alpha]}^{\eta_k\restriction\alpha}:\alpha<\kappa\rangle$ realize the same type over $M$. Since the latter sequence only depends on the coordinates of $\eta_k$, our map $p_k\mapsto\eta_k$ is injective. Therefore, $|\gs^\kappa(M)|\leq\mu^\kappa$. Since $M$ is arbitrary, $\gs_\lambda^\kappa\leq\mu^\kappa=(\gs_\lambda^1)^\kappa$.
\end{proof}

\begin{lemma}\emph{\cite[Theorem 3.1]{bon3.1}}\mylabel{bonlem4}{Lemma \thetheorem}
Let ${\bf K}$ be an AEC with $\lambda$-$AP$ and $\lambda$-$JEP$. Let $\kappa\geq1$ be a cardinal. Then $(\gs^1_\lambda)^\kappa=\gs^\kappa_\lambda$ where $\lambda$-$JEP$ is used in the ``$\leq$'' direction.
\end{lemma}
\begin{proof}
Combine \ref{bonlem1}, \ref{bonlem2} and \ref{bonlem3}.
\end{proof}
\begin{proof}[Proof of \ref{bonthm}]
Let $\kappa=|\alpha|$. By \ref{bonlem4}, $\lambda=\lambda^{\kappa}=(\gs^1_\lambda)^\kappa\geq\gs^\kappa_\lambda$ which shows that ${\bf K}$ is $\kappa$-stable in $\lambda$. By reordering the index $\kappa$, ${\bf K}$ is $\alpha$-stable in $\lambda$.
\end{proof}

\begin{theorem}\emph{\cite[Corollary 6,4]{GV06b}}\mylabel{gvstab}{Theorem \thetheorem}
Let $\mu$ be an infinite cardinal and ${\bf K}$ be an AEC with $AP$. If ${\bf K}$ is $\mu$-tame and stable in $\mu$, then ${\bf K}$ is stable in all $\lambda=\lambda^\mu$. 
\end{theorem}
The original proof proceeds semantically and we will give a syntactic proof in Section \ref{synspl}, allowing stability over sets which can be of size $<\ls$.

\section{Stability and no order property}
To find the upper bound of the first stability cardinal in stable complete first-order theories, one possible way is to establish:
\begin{fact}[Shelah]
Let $T$ be a complete first-order theory. The following are equivalent:
\begin{enumerate}
\item $T$ is stable.
\item For all $\lambda=\lambda^{|T|}$, $T$ is stable in $\lambda$. 
\item $T$ has no (syntactic) order property of length $\omega$.
\end{enumerate}
\end{fact}
$2^{|T|}$ is an upper bound for the first stability cardinal. Notice that in showing (3), compactness is used to stretch the order property to arbitrary length. In AECs, we can use the Hanf number to bound the (Galois) order property length. The following definition is based on \cite[Definition 4.3]{sh394} and \cite[Definition 4.3]{s5}:
\begin{definition}\mylabel{fsdef}{Definition \thetheorem}
Let $\mu$ be an infinite cardinal, $\alpha\geq2$ and $\beta\geq1$ be ordinals.
\begin{enumerate}
\item $K$ has the \emph{$\beta$-order property of length $\mu$} if there exists some $\langle a_i:i<\mu\rangle\subseteq M\in K$ such that $l(a_i)=\beta$, and for $i_0<i_1<\mu$, $j_0<j_1<\mu$, $\gtp(a_{i_0}a_{i_1}/\emptyset,M)\neq\gtp(a_{j_1}a_{j_0}/\emptyset,M)$.
\item $K$ has the \emph{$(<\alpha)$-order property of length $\mu$} if there is a $\beta<\alpha$ witnessing (1).
\item $K$ has the \emph{$(<\alpha)$-order property} if for all $\mu$, $K$ has the $(<\alpha)$-order property of length $\mu$. In other words, if we fix $\mu$, we can find a suitable $\beta_\mu$ witnessing (1).
\item $K$ has the \emph{no $(<\alpha)$-order property} if (3) fails. In other words, for each $\beta<\alpha$, there is an upper bound to the length of the $\beta$-order property. We omit $(<\alpha)$ if $\alpha=\omega$.
\end{enumerate}
\end{definition}
The above definition works fine if one wants an abstract generalization of the order property from the first-order version, in which case the length can be fixed at $\omega$. However, in AECs, it is hard to construct long well-ordered sets without breaking stability or raising $\ls$. We propose the following definition instead:
\begin{definition}\mylabel{fsdef2}{Definition \thetheorem}
In \ref{fsdef}, we replace all occurences of ``order property'' by ``order property*'' if we also allow sequences indexed by linear orders instead of well-orderings. For example in (1), we say $K$ has the \emph{$\beta$-order property* of length $\mu$} if there exist some linear order $I$, some $\langle a_i:i\in I\rangle\subseteq M\in K$ such that $|I|=\mu$, $l(a_i)=\beta$, and for $i_0<i_1$ in $I$, $j_0<j_1$ in $I$, $\gtp(a_{i_0}a_{i_1}/\emptyset,M)\neq\gtp(a_{j_1}a_{j_0}/\emptyset,M)$. When $\mu$ is omitted, we mean for all $\mu$, there is a linear order $I$ of cardinality $\mu$ witnessing the $\beta$-order property* of length $\mu$.
\end{definition} 

In the following proposition, item (1) applies Morley's method \cite{mor} (see also \cite[Theorem A.3(2)]{bal}). The statement we use is from \cite[Claim 4.6]{sh394} which only hinted at the proof of the Hanf number for arbitrarily large models \cite[VII.5]{sh90}. We add more details and explain how to adapt that proof. The proof of item (3) adapts the proof from \cite[Fact 5.13]{BGKV}.
\begin{proposition}\mylabel{fsprop2}{Proposition \thetheorem}
Let ${\bf K}$ be an $AEC$, $\beta\geq2$ be an ordinal.
\begin{enumerate}
\item If for all $\mu<\beth_{(2^{<(\ls^++|\beta|)})^+}$, ${\bf K}$ has the $(<\beta)$-order property of length $\mu$, then ${\bf K}$ has the $(<\beta)$-order property (and the $(<\beta)$-order property*).
\item If for all $\mu<\beth_{(2^{<(\ls^++|\beta|)})^+}$, ${\bf K}$, ${\bf K}$ has the $(<\beta)$-order property* of length $\mu$, then ${\bf K}$ has the $(<\beta)$-order property* (and $(<\beta)$-order property).
\item If ${\bf K}$ is $(<\beta)$-stable (in some $\lambda\geq\ls+|\beta|$), then there is $\mu<\beth_{(2^{<(\ls^++|\beta|)})^+}$ such that ${\bf K}$ has no $(<\beta)$-order property* (and thus no $(<\beta)$-order property) of length $\mu$. 
\end{enumerate}
\end{proposition}
\begin{proof}[Proof sketch]
\begin{enumerate}
\item We adapt the usual Hanf number argument. Suppose ${\bf K}$ has the $(<\beta)$-order property and we fix $\gamma<\beta$ such that ${\bf K}$ has the $\gamma$-order property. By Shelah's Presentation Theorem, we may write $K=PC(T,\Gamma_1,\llk)$ for some first-order theory $T$ in $L\supseteq\llk$ and some sets of $L$-types $\Gamma_1$. Now we refer to the construction of \cite[VII Theorem 5.3]{sh90} or \cite[Chapter 2 Theorem 6.35]{Gbook}. For each $\alpha<(2^{\ls+|\gamma|})^+$, instead of defining $F^{\mathfrak{B}}(\alpha)$ to be some $M\in K$ of size $\beth_\alpha$, we demand it to be the witness of the $\gamma$-order property of length $\beth_\alpha$ (we can also add another function $F_1^{\mathfrak{B}}(\alpha,\cdot)$ to enumerate the elements of $F^{\mathfrak{B}}(\alpha)$). At the end of the construction (which uses Erd\H{o}s-Rado Theorem
), we obtain an $L$-indiscernible sequence (of $\gamma$-tuples) $\langle a_i:i<\omega\rangle$ such that for $n<\omega$, $i_1<i_2<\cdots<i_n<\omega$, the first-order type of $a_{i_1}\dots a_{i_n}$ is realized by some $d_1\dots d_n$ that witness the order property. This induces an $L$-isomorphism between $EM(a_{i_1}\dots a_{i_n})\cong EM(d_1\dots d_n)$. Its reduct to $\llk$ is also an isormophism between $EM(a_{i_1}\dots a_{i_n})\restriction \llk\cong EM(d_1\dots d_n)\restriction \llk$. Since the right-hand-side witnesses the order property in $K$, so is the left-hand-side. The same argument applies when the indiscernible sequence is stretched to arbitrary length (or any linear order).
\item The same proof of (1) goes through because Erd\H{o}s-Rado Theorem 
applies to linear orders (actually any sets) besides well-orderings. 
\item Otherwise by (1)(2), ${\bf K}$ has the $(<\beta)$-order property*. For any infinite cardinal $\lambda\geq\ls+|\beta|$, let $I_\lambda$ be a linear order of size $>\lambda$ such that it has a dense suborder $J_\lambda$ of size $\lambda$. We stretch the indiscernible $\langle a_i:i<\omega\rangle$ in (2) to be indexed by $I_\lambda$. Pick $i<i'$ in $I_\lambda$, we can find $k$ in $J_\lambda$ such that $i<j<i'$. Then $\gtp(a_ia_j/\emptyset)\neq\gtp(a_{i'}a_j/\emptyset)$ are distinct by the order property, which means $\gtp(a_i/a_j)\neq\gtp(a_{i'}/a_j)$. This shows that ${\bf K}$ is $(<\beta)$-unstable in $\lambda$. As $\lambda$ is arbitrary, ${\bf K}$ is $(<\beta)$-unstable above $\ls+|\beta|$.
\end{enumerate}
\end{proof}
\begin{remark}
\begin{itemize}
\item If we have a specific $K$ in mind, we may replace $(2^{\ls})^+$ by $\delta(\ls,\kappa)$ where $\kappa\defeq\card{\Gamma_1}\leq2^{\ls}$ in Shelah's Presentation Theorem. 
\item We used Galois types over the empty set in proving (2). The same proof goes through if we require the domain of the types to be some fixed (nonempty) model.
\end{itemize}
\end{remark}
\section{Lower bound for stability and no order property*}
From \ref{fsprop2}, we saw that a stable AEC cannot have the order property* of length $\mu$ for some $\mu<\han$. Our goal is to show that this is also a lower bound for no order property* as well as stability. 
\begin{proposition}\mylabel{fsprop}{Proposition \thetheorem}
Let $\lambda$ be an infinite cardinal and $\alpha$ be an ordinal with $\lambda\leq\alpha<(2^\lambda)^+$. Then there is a stable AEC ${\bf K}$ such that $\ls=\lambda$, ${\bf K}$ has the order property* of length up to $\beth_{\alpha}(\lambda)$ and is unstable anywhere below $\beth_\alpha(\lambda)$. Moreover, ${\bf K}$ has $JEP$, $NMM$ and $(<\al)$-tameness but not $AP$. 
\end{proposition}
The proof of \ref{fsprop} will come after \ref{fscor3} and use \ref{fslem} below.
\begin{lemma}\emph{\cite[VII Theorem 5.5(6)]{sh90}}\mylabel{fslem}{Lemma \thetheorem}
Let $\lambda$ be an infinite cardinal and $\alpha$ be an ordinal with $\lambda\leq\alpha<(2^\lambda)^+$. There is a stable AEC ${\bf K_1}$ such that $\lss=\lambda$ and for all $M\in { K_1}$, $M$ is well-ordered of order type at most $\alpha$. Moreover, ${\bf K_1}$ has $AP$, $JEP$ and is $(<\al)$-tame. 
\end{lemma}
\begin{proof}
Let $L_1\defeq\{<,P_i:i<\lambda\}$ where $P_i$ are unary predicates. $T_1$ requires $<$ to be a linear order. For $\beta<\alpha$, pick distinct subsets $S_\beta\subseteq\lambda$. Given an element $x$, we define its \emph{type} by the set of indices $i$ such that $P_i(x)$. We require that each element is characterized by its type and the only possible types are among $\{S_\beta:\beta<\alpha\}$. If $\beta<\alpha$ and $x,y$ have types $S_\beta$ and $S_\alpha$ respectively, then we stipulate that $x<y$. These will be coded in the collection of types $\Gamma_1$.

More precisely, for $\beta<\gamma<\alpha$, $S\subseteq\lambda$,
\begin{align*}
p_{\beta,\gamma}(x,y)\defeq&\{\neg (x<y)\}\cup\{P_i(x):i\in S_\beta\}\cup\{\neg P_i(x):i\not\in S_\beta\}\\&\cup\{P_i(y):i\in S_\gamma\}\cup\{\neg P_i(y):i\not\in S_\gamma\}\\
p_S(x)\defeq&\{P_i(x):i\in S\}\cup\{\neg P_i(x):i\not\in S\}\\
\Gamma_1\defeq&\{p_{\beta,\gamma}(x,y):\beta<\gamma<\alpha\}\cup\{p_{S'}:S'\subseteq\lambda\text{ such that for $\beta<\alpha$, }S'\neq S_\beta\}\\&\cup\{x\neq y\wedge P_i(x)\leftrightarrow P_i(y):i<\lambda\}
\end{align*}

Let ${\bf K_1}\defeq EC(T_1,\Gamma_1)$ ordered by substructures. Notice that $|L(T_1)|=\lambda$ and $\Gamma_1=2^\lambda$. By replacing ${\bf K_1}$ by $({\bf K_1})_{\geq\lambda}$, we may assume $\lss=\lambda$. Then $M_1=\alpha$ is the maximal model (every model can be extended to an isomorphic copy of it) where for $\beta\in M_1$, $i<\lambda$, $P_i^{M_1}(\beta)$ iff $i\in S_\beta$. Hence ${\bf K_1}$ satisfies $AP$, $JEP$ (but not $NMM$). As ${\bf K_1}$ has the maximal model of size $|\alpha|\leq2^\lambda$, ${\bf K_1}$ is trivially stable in $\geq(2^\lambda)^+$. It is $(<\al)$-tame because types are decided by the $P_i$'s they belong to.

\end{proof}
\begin{remark}\mylabel{fsrmk}{Remark \thetheorem}
By \cite[VII Theorem 5.5(2)]{sh90}, for any $\lambda$ and $\kappa\leq2^\lambda$, $\delta(\lambda,\kappa)\leq(2^\lambda)^+$, so the threshold $(2^\lambda)^+$ cannot be improved. If we restrict $1\leq\kappa<2^\lambda$, then for $\alpha<\lambda^+$, we can still define ${\bf K_1}$ to be well-orderings of type at most $\alpha$, where models are ordered by initial segments. Then we get a lower bound $\delta(\lambda,\kappa)\geq\lambda^+$. But the above proof does not go through because it requires $|\Gamma|=2^\lambda$.
\end{remark}

Using \ref{fslem}, we are able to answer Boney's conjecture in \ref{bonresrmk}. We will complete the proof of \ref{fsprop} after the proof of \ref{fscor3}.

\begin{corollary}\mylabel{fscor3}{Corollary \thetheorem}
Under GCH, the $\lambda$-$JEP$ assumption in \ref{bonlem2.5} and \ref{bonlem2} is not necessary. If $2^\al=\aleph_{\omega_1}$ and $\lambda=\al$, then $\lambda$-$JEP$ is necessary in \ref{bonlem2.5}. If $2^\al=\aleph_1\eqdef\lambda$ and $2^{\aleph_1}=\aleph_{\omega_2}$, then $\lambda$-$JEP$ is necessary in \ref{bonlem2}.
\end{corollary}
\begin{proof}
By \ref{bonlem1}, we may assume $\kappa<\lambda$. Suppose ${\bf K}$ is $\kappa$-stable in $\lambda$, then \ref{bonlem2.5} and \ref{bonlem2} are always true: for \emph{all} $\nr{M}=\lambda$, $\gs^\kappa_\lambda=\lambda=|\gs^\kappa(M)|$. Also, by taking sequences of length $\kappa$ from $|M|$ (which give the algebraic types), we have $\gs^\kappa_\lambda\geq|\gs^\kappa_\lambda(M)|\geq\lambda^\kappa\geq(\gs^1_\lambda)^\kappa$ where the last inequality is by stability.

Therefore, we may further assume $\kappa$-instability in $\lambda$, witnessed by $M$. Suppose GCH holds, then $|\gs^\kappa(M)|\geq\lambda^+=2^\lambda=2^{\lambda+\kappa}\geq\gs^\kappa_\lambda\geq|\gs^\kappa(M)|$. Also, $\gs^\kappa_\lambda\geq|\gs^\kappa(M)|\geq\lambda^+=2^\lambda=(2^\lambda)^\kappa\geq(\gs^1_\lambda)^\kappa$. Hence it is consistent that the lemmas are always true (regardless of $\lambda$-$JEP$ or $\cf(\gs_\lambda^\kappa)$).

We show that \ref{bonlem2.5} can be false: let $\lambda=\al$, $\kappa=1$ and suppose $2^\al=\aleph_{\omega_1}$ (which is consistent by Easton Theorem, see \cite[Theorem 15.18]{jech}). Define $L_1,T_1$ as in \ref{fslem}, pick $\aleph_\omega$ many distinct subsets of $\al$, say $\{S_\beta:\beta<\aleph_\omega\}$. We allow $\beta<\gamma<\aleph_{\omega}$ when constructing $p_{\beta,\gamma}$. Now for $n<\omega$, define $\Gamma_1^n$ to be the same as $\Gamma_1$ except that $\alpha$ is replaced by $\aleph_n$. Define ${\bf K^n}\defeq EC(T_1,\Gamma_1^n)$ and ${\bf K}$ be the disjoint union of all ${\bf K^n}$ (adding $\al$-many predicates and stipulate that no two elements belong to different predicates --- this destroys $\al$-$JEP$). Notice that ${\bf K}$ is still a $EC$ class where the language has size $\al$ and whose models omit $2^\al$-many types. For any $M\in K$, $M\in K^n$ for some $n<\omega$. Since the unique maximal model in ${\bf K^n}$ has size $\aleph_n$, $|\gs(M)|\leq\aleph_n<\aleph_\omega=\sup_{n<\omega}\aleph_n=\gs_{\al}^1$ where the last equality is due to the fact that distinct two elements in a model satisfy distinct subsets of $\{P_i:i<\al\}$. Hence $\gs_\al^1$ (and similarly all $\gs_{\aleph_n}^1$, $n<\omega$) is not attained by any model. $\cf(\gs_\al^1)=\al$ which satisfies the hypothesis. 

We show that \ref{bonlem2} can also be false: let $\lambda=\aleph_1$, $\kappa=\al$ and suppose $2^\al=\aleph_1$ and $2^{\aleph_1}=\aleph_{\omega_2}$ (which is consistent by Easton Theorem). This time we pick $\aleph_\omega$ many distinct subsets from $\aleph_1$ rather than from $\al$. Form ${\bf K}$ as above which is an $EC$ class whose language has size $\aleph_1$ and whose models omit $2^{\aleph_1}$-many types. Now $(\gs_\lambda^1)^\kappa=(\gs_{\aleph_1}^1)^\al={\aleph_\omega}^\al>\aleph_\omega$. On the other hand, $\gs_\lambda^\kappa=\gs_{\aleph_1}^\al\leq\sup_{n<\omega}\sup\{|\gs^\al(M)|:M\in (K^n)_{\aleph_1}\}\leq\sup_{n<\omega}{\aleph_n}^\al=\sup_{n<\omega}(2^\al\cdot\aleph_n)=\sup_{n<\omega}\aleph_1\cdot\aleph_n=\aleph_\omega<(\gs_\lambda^1)^\kappa$ where the second equality is a speical case of the Hausdorff formula \cite[Equation 5.22]{jech}. 
\end{proof}
\begin{proof}[Proof of \ref{fsprop}]
Fix $\alpha$ as in \ref{fslem}, we use ${\bf K_1}$ to build ${\bf K}$ as follows: let $L\defeq L_1\cup\{E,Q,R,f,g,c\}$ where $E,Q$ are binary predicates, $R$ is a ternary predicate, $f$ is a unary function, $g$ is a symmetric binary function and $c$ is a constant. Thus we have $\card{L}=\lambda$. A model in ${\bf K}=({\bf K_0},{\bf K_1},{\bf K_2})$ has three sorts $(M_0,M_1,M_2)$. $M_1$ is in $K_1$, $M_2$ will take care of $NMM$ while $M_0$ is the iterated power sets of $M_1$. In details, we require: 
\begin{enumerate}
\item $M_1\in EC(T_1,\Gamma_1)$ as in \ref{fslem}. We identify it as an ordinal $\leq \alpha$. 
\item $M_2$ is an infinite model of the theory of pure equality.
\item If $xEy$, then $x$ is in $M_0\cup M_1$ while $y$ is in $M_0$. $E$ also satisfies the extensionality axiom.
\item The first argument of $Q$ is in $M_1$. We write $Q_i(\cdot)\defeq Q(i,\cdot)$ and abbreviate $Q_i$ the elements $x$ in $M_0$ with $Q(i,x)$. For limit ordinal $\sigma$ in $M_1$, we require $Q_\sigma=\bigcup_{i<\sigma}Q_i$.
\item $f$ is the rank function from $M_0$ to $M_1$ such that each $x$ is sent to the smallest $i$ with $Q_i(x)$. If $x\in y$, then $f(x)<f(y)$.
\item $R$, $g$ and $c$ code the total order of $M_0$: we define $R(\beta,x,y)$ and $g(x,y)$ as follows:
\begin{enumerate}
\item $\beta$ in $M_1$, $x,y$ in $M_0$
\item If $f(x)\neq f(y)$, then we say $x$ is less than $y$ when $x$ has a smaller rank than $y$. 
\item If $f(x)=f(y)=0$, then $g(x,y)$ is the $<$-least element in the symmetric difference of $x,y$. $R(0,x,y)$ if $g(x,y)E y$, in which case we say $x$ is less than $y$.
\item If $f(x)=f(y)=\beta>0$, then $g(x,y)$ is the least element in the symmetric difference of $x,y$. $R(\beta,x,y)$ if $g(x,y)E y$, in which case we say $x$ is less than $y$.
\item $c$ is the default value for $g(x,y)$ when $x,y$ are in $M_1$ or $M_2$, or when $x,y$ in $M_0$ are equal or have different ranks.
\end{enumerate}
\end{enumerate}
(If we think of subsets as sequences, we are ordering $M_0$ by rank, and then by lexicographical order of each rank.)

We order $(M_0,M_1,M_2)\lk(N_0,N_1,N_2)$ iff for $i\leq2$, $M_i,N_i\in K_i$ and $M_i\subseteq N_i$. Notice that we can describe ${\bf K}$ as some $EC(T,\Gamma)$ where $T$ is a $\forall\exists$ theory, $\card{L(T)}=\lambda$ and $|\Gamma|=2^\lambda$. Also, $\ls=\lambda$ because $\lss=\lambda$ and we can close any set in $M_0$ to a model by adding witnesses for $f,g$ in (5),(6) $\omega$-many times. Therefore ${\bf K}$ is an AEC.

The maximal model in $({\bf K_0},{\bf K_1})$ is $M^*\defeq(V_\alpha(\alpha),\alpha)$, where for $\beta<\alpha$, $Q_\beta^{M^*}\defeq V_{1+\beta}(\alpha)$. $M^*$ witnesses that ${\bf K}$ have $JEP$. ${\bf K}$ is $(<\al)$-tame because elements are determined either by the predicates $P_i$; or their ranks and their own elements. With ${\bf K_2}$, we know that ${\bf K}$ has $NMM$. Since $M^*$ has size $\beth_\alpha(\lambda)$, $({\bf K_0},{\bf K_1})$ is trivially stable in $\geq\beth_\alpha(\lambda)$. As ${\bf K_2}$ is trivially stable everywhere, ${\bf K}$ is stable in $\geq\beth_\alpha(\lambda)$. We now show instability in $<\beth_\alpha(\lambda)$ and the order property of lengths up to $\beth_\alpha(\lambda)$.

For instability, $(\mathcal{P}(\lambda),\lambda,\omega)$ witnesses that ${\bf K}$ is unstable between $[\lambda,2^\lambda)$. Consider the maximal model $M^*$ above, for each $\beta<\alpha$, $\card{Q_\beta^{M^*}}=\beth_{1+\beta}(\lambda)$. By item (6) above, $Q_\beta^{M^*}$ can distinguish all elements in $Q_{\beta+1}^{M^*}$, so $({\bf K_0},{\bf K_1})$, and hence ${\bf K}$ is unstable in $\big[\beth_{1+\beta}(\lambda),\beth_{1+\beta+1}(\lambda)\big)$. Therefore, we have instability in $\big[\lambda,\beth_{\alpha}(\lambda)\big)$. 

For the order property*, we apply the order in (6) to $V_\alpha(\alpha)$ of the maximal model. In details: let $a$ less than $b$ while $c$ less than $d$. If $a,b$ is mapped to $d,c$ respectively, then $f(d)=f(a)\leq f(b)=f(c)\leq f(d)$. It cannot happen to rank 0 because their elements are well-ordered in $M_1$. Since $a$ is less than $b$, $g(a,b)\in b$. Then by mapping $g(d,c)(=g(c,d))\in c$ which shows that $d$ is less than $c$, contradiction. As $|V_\alpha(\alpha)|=\beth_\alpha(\lambda)$, we have the order property* of length $\beth_\alpha(\lambda)$.

$({\bf K_0},{\bf K_1})$ does not have $AP$: Pick an element $x$ from $(V_\alpha(\alpha),\alpha)$ which contains $\geq\lambda^+$ elements. Close $x$ to a substructure $N$ of size $\lambda$, then there is $yEx$ in $(V_\alpha(\alpha),\alpha)$ but $y$ is not in $N$. $N$ can be included in $(V_\alpha(\alpha),\alpha)$ such that $x$ is mapped to $(x-\{y\})$. Suppose the following amalgam exists:
\begin{center}
\begin{tikzcd}
{x\in(V_\alpha(\alpha),\alpha)} \arrow[r, dotted] & x\in W                                                       \\
x\in N \arrow[r] \arrow[u]                             & {x-\{y\}\in(V_\alpha(\alpha),\alpha)} \arrow[u, "t", dotted]
\end{tikzcd}
\end{center}
Without loss of generality, we may assume the top dotted arrow is identity (hence we can write the image of $x$ to be $x$ itself). Since $(V_\alpha(\alpha),\alpha)$ is maximal, $W=(V_\alpha(\alpha),\alpha)$. Therefore, $t\in\oop{Aut}((V_\alpha(\alpha),\alpha))$. By an induction argument, $t$ must be the identity (which boils down to the fact that $\alpha$, the maximal model in ${\bf K_1}$, is rigid). From the right dotted arrow, $x-\{y\}$ would be mapped to $x$, which is a contradiction.
\end{proof}
\begin{remark}
\begin{enumerate}
\item \mylabel{fsrmk2}{Remark \thetheorem} One way to save $AP$ is to redefine $\lk$ by the $E$-transitive closure, but it raises $\ls$ to $V_{\alpha-1}(\lambda)$. In this case, instability and the order property* length are up to $2^{\ls}$.
\item Our total order is ill-founded: $\alpha ,  \alpha-\{0\} ,  \alpha-\{0,1\} ,  \dots$ form an infinite descending sequence in $Q_0^{M^*}$. It is not clear how to extract a witness to the order property of length $>2^\lambda$. We will see in \ref{fsprom} that we can refine our example to have the order property at least up to $\beth_{\alpha-3}(\lambda)$.
\end{enumerate}
\end{remark}
\begin{corollary}\mylabel{fscor}{Corollary \thetheorem}
\begin{enumerate}
\item For stable AECs, the Hanf number for the order property* length is exactly $\han$.
\item The Hanf number for stability is at least $\han$. In other words, let $\lambda\geq\al$ and $\mu<\beth_{(2^\lambda)^+}$, there is a stable AEC ${\bf K}$ such that $\ls=\lambda$ and the first stability cardinal is greater than $\mu$.
\end{enumerate}
\end{corollary}
\begin{proof}
Combine \ref{fsprop} and \ref{fsprop2} and range $\alpha$ in $[\lambda,(2^\lambda)^+)$. 
\end{proof}
In the next two sections, we develop the machinery to show that the lower bound in (2) is tight, based on the arguments in \cite{s5}. Then we conclude: our example witnesses that the bound for the order property is also tight.

\section{Galois Morleyization and syntactic order property}
Galois Morleyization is a way to capture tameness syntactically by adding infinitary predicates. First recall the definition of tameness:
\begin{definition}\mylabel{tame2}{Definition \thetheorem}
Let $\kappa$ be an infinite cardinal. 
\begin{enumerate}
\item Let $p=\gtp(a/A,N)$ where $a=\langle a_i:i<\alpha\rangle$ may be infinite, $I\subseteq\alpha$, $A_0\subseteq A$. We write $l(p)\defeq l(a)$, $p\restriction A_0\defeq \gtp(a/A_0,N)$, $a^I=\langle a_i:i\in I\rangle$ and $p^I\defeq\gtp(a^I/A,N)$.
\item $K$ is \emph{$(<\kappa)$-tame for $(<\alpha)$-types} if for any subset $A$ in some model of $K$, any $p\neq q\in\gs^{<\alpha}(A)$, there is $A_0\subseteq A$, $|A_0|<\kappa$ with $p\restriction A_0\neq q\restriction A_0$. We omit $(<\alpha)$ if $\alpha=2$.
\item $K$ is \emph{$(<\kappa)$-short} if for any $\alpha\geq2$, any subset $A$ in some model of $ K$, $p\neq q\in\gs^{<\alpha}(A)$, there is $I\subseteq\alpha$, $\card{I}<\kappa$ with $p^I\neq q^I$. 
\item \emph{$\kappa$-tame} means $(<\kappa^+)$-tame. Similarly for shortness.
\end{enumerate}
\end{definition}
Now we construct Galois Morleyization:
\begin{definition}\cite[Definitions 3.3, 3.13]{s5}
Let $\kappa$ be an infinite cardinal and $K$ be an AEC in a (finitary) language $L$. The \emph{$(<\kappa)$-Galois Morleyization of $K$} is a class $\hat{K}$ of structures in a language $\hat{L}$ such that:
\begin{enumerate}
\item $\hat{L}$ is a $(<\kappa)$-ary language. For convenience we may require $L\subseteq\hat{L}$.
\item For each $p\in\gs^{<\kappa}(\emptyset)$, we add a predicate $R_p$ of length $l(p)$ to $\hat{L}$.
\item For each $M\in K$, we define $\hat{M}\in \hat{K}$ with $|M|=|\hat{M}|$. For $p\in\gs^{<\kappa}(M)$, $a\in |\hat{M}|^{l(p)}$, $\hat{M}\vDash R_p[a]$ iff $a\vDash p$ in $K$. Extend the definition to quantifier-free formulas of $\hat{L}_{\kappa,\kappa}$.
\item The \emph{$(<\kappa)$-syntactic type} of $a\in|\hat{M}|^{<\kappa}$ over $A\subseteq|\hat{M}|$ is defined by $\tp_{\text{qf-}\hat{L}_{\kappa,\kappa}}(a/A;\hat{M})$, namely the quantifier-free formulas of $\hat{L}_{\kappa,\kappa}$ over $A$ that $a$ satisfies. We will abbreviate it as $\tp_\kappa(a/A;\hat{M})$. 
\item For $\hat{M},\hat{N}\in \hat{K}$, we order $\hat{M}\leq_{\bf \hat{K}}\hat{N}$ iff $M\lk N$. We will omit the subscripts. 
\end{enumerate}
\end{definition}
\begin{remark}
\begin{enumerate}
\item If we allow AECs to have infinitary languages, we can view $\hat{K}$ as an AEC. 
\item The above is well-defined even for AECs without $AP$, but readers can assume the existence of a monster model $\mn$ for convenience.
\item $\card{\hat{L}}=\card{L}+\card{\gs^{<\kappa}(\emptyset)}\leq 2^{<(\kappa+\ls^+)}$.
\end{enumerate}
\end{remark}
The following justifies the definition of Galois Morleyization in tame AECs:
\begin{proposition}\emph{\cite[Corollary 3.18(2)]{s5}}\mylabel{gmprop}{Proposition \thetheorem}
Let $K$ be an AEC and $\hat{K}$ be its $(<\kappa)$-Galois Morleyization. For each $p=\gtp(b/A;M)\in\gs(A)$, define $p_\kappa\defeq\tp_\kappa(b/A;\hat{M})$ to be its $(<\kappa)$-syntactic version. Then $K$ is $(<\kappa)$-tame iff $p\mapsto p_\kappa$ is a 1-1 correspondence.
\end{proposition}
\begin{proof}
$\Rightarrow$: The map is well-defined because Galois types are finer than syntactic types. It is a surjection by construction. Suppose $p\neq q\in\gs(A)$, by $(<\kappa)$-tameness we may assume the domain $A$ has size $<\kappa$. Let $b\vDash p$ and $b'\vDash q$. Then $bA$ and $b'A$ satisfy different types in $\gs^{<\kappa}(\emptyset)$, say $r$ and $s$ repsectively. Thus $bA\vDash R_r\wedge\neg R_s$ while $b'A\vDash R_s\wedge\neg R_r$.

$\Leftarrow$: Suppose $p=\gtp(b/A;M)\neq q=\gtp(b'/A;M')$. Then $p_\kappa\neq q_\kappa$ and we can find $r\in\gs^{<\kappa}(\emptyset)$ and (a suitable enumeration of) $A_0\subseteq A$ such that $M\vDash R_r[b;A_0]$ but $M'\vDash \neg R_r[b';A_0]$. This means $bA_0\vDash r$ while $b'A_0\not\vDash r$. Hence $\gtp(b/A_0;M)\neq\gtp(b'/A_0;M')$ witnessing $(<\kappa)$-tameness.
\end{proof}
\begin{remark}
There is a stronger version assuming $(<\kappa)$-shortness in \cite[Corollary 3.18(1)]{s5} but we have no use of it here.
\end{remark}

We now define an infinitary version of the syntactic order property:
\begin{definition} Let $2\leq\alpha\leq\kappa$ and $1\leq\beta<\kappa$ be ordinals. 
\begin{enumerate}
\item \cite[Definition 4.2]{s5} In \ref{fsdef}, we replace all occurences of ``order property'' by ``syntactic order property'' while requiring the condition in (1) there be: there exist some $\langle a_i:i<\mu\rangle\subseteq \hat{M}\in \hat{K}$ and some quantifier-free $\hat{L}_{\kappa,\kappa}$ formula $\phi(x,y)$ such that $l(a_i)=\beta$, and for $i,j<\mu$, $i<j$ iff $\hat{M}\vDash\phi[a_i,a_j]$.
\item As in \ref{fsdef2}, we define syntactic order property* if in (1) we allow the index set to be a linear order $I$ with $|I|=\mu$, instead of being a well-ordering $\mu$.
\end{enumerate}
\end{definition}

The following links the (Galois) order property in $K$ with the syntactic order property in $\hat{K}$. 
\begin{proposition}\mylabel{fsprop3}{Proposition \thetheorem}
Let $\kappa$ be an infinite cardinal and $\hat{K}$ be $(<\kappa)$-Galois Morleyization of $K$. Let $1\leq\beta<\kappa$ be an ordinal and $M\in K$. Let $\lambda,\mu$ be infinite cardinals and $\chi\defeq\card{\gs^{\beta+\beta}(\emptyset)}$.
\begin{enumerate}
\item \emph{\cite[Proposition 4.4]{s5}}
\begin{enumerate}
\item If $\hat{M}$ has the syntactic the $\beta$-order property of length $\mu$, then $M$ has the $\beta$-order property of length $\mu$.
\item If $M$ has the $\beta$-order property of length $\mu$ for some $\mu\geq (2^{\lambda+\chi})^+$, then $\hat{M}$ has the syntactic $\beta$-order property of length $\lambda^+$.
\end{enumerate}
\item
\begin{enumerate}
\item If $\hat{M}$ has the syntactic $\beta$-order property* of length $\mu$, then $M$ has the $\beta$-order property* of length $\mu$.
\item If $M$ has the $\beta$-order property* of length $\mu$ for some $\mu\geq (2^{\lambda+\chi})^+$, then $\hat{M}$ has the syntactic $\beta$-order property* of length $\lambda^+$.
\end{enumerate}
\item The following are equivalent:
\begin{enumerate}
\item $K$ has the $\beta$-order property.
\item $K$ has the $\beta$-order property*.
\item $\hat{K}$ has the syntactic $\beta$-order property.
\item $\hat{K}$ has the syntactic $\beta$-order property*.
\end{enumerate}
\end{enumerate}
\end{proposition}
\begin{proof}
\begin{enumerate}
\item (a) is true because Galois types are finer than syntactic types. For (b): suppose $\langle a_i:i<\mu\rangle$ witnesses the $\beta$-order property. By Erd\H{o}s-Rado Theorem
, we have $\mu\rightarrow(\lambda^+)^2_{2^\chi}$. Apply this to $(i<j)\mapsto\gtp(a_ia_j/\emptyset;M)$ and then on $(j<i)\mapsto\gtp(a_ja_i/\emptyset;M)$. We can find $\langle b_i:i<\lambda^+\rangle$ subsequence of $\langle a_i:i<\mu\rangle$, $p\neq q\in\gs^{\beta+\beta}(\emptyset)$ such that for $i<j<\lambda$, $\gtp(b_ib_j/\emptyset;M)=p$ and $\gtp(b_jb_i/\emptyset;M)=q$. We may choose $R_p$ (or $R_q$) to witness the $\beta$-syntactic order property.
\item The same proof goes through, because Erd\H{o}s-Rado Theorem 
applies to linear orders too. 
\item (1) gives (a)$\Leftrightarrow$(c) while (2) gives (b)$\Leftrightarrow$(d). (a)$\Leftrightarrow$(b) is by \ref{fsprop2}(1),(2).
\end{enumerate}
\end{proof}
\begin{definition}
Let $\kappa$ be an infinite cardinal, $2\leq\alpha\leq\kappa$ and $1\leq\beta<\kappa$ be ordinals, $\hat{K}$ be the $(<\kappa)$-Galois Morleyization of $K$. Then 
\begin{enumerate}
\item For $\mu\geq\ls+|\beta|$, $\hat{K}$ is \emph{$\beta$-syntactically stable in $\mu$} if $$\card{\{p_\kappa:p\in\gs^{\beta}(A;M),A\subseteq|M|,|A|\leq\mu,M\in K\}}\leq\mu.$$
\item For $\mu\geq\ls+|\alpha|$, $\hat{K}$ is \emph{$(<\alpha)$-syntactically stable in $\mu$} if $$\card{\{p_\kappa:p\in\gs^{<\alpha}(A;M),A\subseteq|M|,|A|\leq\mu,M\in K\}}\leq\mu.$$
\end{enumerate}
\end{definition}
\begin{corollary}Let $\beta\geq1$ be an ordinal and $\mu\geq\ls+|\beta|$ be a cardinal.
\begin{enumerate}
\item \emph{\cite[Fact 4.9]{s5}} If $K$ has the $\beta$-order property, then $K$ is not $\beta$-stable in $\mu$. If also $\beta<\kappa$, then $\hat{K}$ is not $\beta$-(syntactically) stable in $\mu$. 
\item The same conclusion holds when $K$ has the $\beta$-order property*.
\end{enumerate}
\end{corollary}
\begin{proof}
By \ref{fsprop3}, either assumption gives the syntactic $\beta$-order property. This implies $\beta$-syntactic instability in $\mu$, using the proof of \ref{fsprop2} (in particular replace $F^{\mathfrak{B}}(\alpha)$ by a witness of the syntactic $\beta$-order property of length $\beth_\alpha$).
\end{proof}
\section{Shelah's stability theorem}
We will connect syntactic stability with no syntactic order property. The original result due to Shelah was in a more general context but only proof sketches were given. Vasey \cite[Fact 4.10]{s5} applied it to AECs without a complete proof so we write out all the details. We will also remove the requirement that the order property length be a successor (which was hinted in \cite[Exercise 1.22]{sh300a}).
\begin{theorem}\emph{\cite[V.A. Theorem 1.19]{sh300a}}\mylabel{fsstthm}{Theorem \thetheorem}
Let $\chi\geq2^{<(\kappa+\ls^+)}$, $\mu$ be an infinite cardinal such that $\mu=\mu^{\chi}+2^{2^{<\chi}}$. Suppose $\hat{K}$ does not have the $(<\kappa)$-syntactic order property of length $\chi$, then $\hat{K}$ is $(<\kappa)$-syntactically stable in $\mu$. 
\end{theorem}
The proof will be given after \ref{fsstlem2}. Before that we state some relevant definitions and lemmas.
\begin{definition}\mylabel{syndef}{Definition \thetheorem}
Let $\kappa$ be an infinite cardinal, $\Pi$ be a set of quantifier-free formulas of $\hat{L}_{\kappa,\kappa}$ over $A$, and $p_\kappa$ be a $(<\kappa)$-syntactic type over $A$. We say \emph{$p_\kappa$ splits over $\Pi$} if there are $\phi(x;b),\neg\phi(x;c)\in p_\kappa$ such that for any $\hat{M}$ containing $b,c$ and the parameters from $\Pi$, any $\psi(y;d)\in\Pi$ with $l(y)=l(b)=l(c)$, we have $\hat{M}\vDash\psi[b;d]\Leftrightarrow \hat{M}\vDash\psi[c;d]$ (the choice of $\hat{M}$ does not matter because its interpretation of $R_p$ is external).

If we require the witnesses $\phi(x;b),\neg\phi(x;c)$ to be from a fixed formula $\phi(x;y)$, then we say \emph{$p_\kappa$ $\phi$-splits over $\Pi$}.
\end{definition}
\begin{lemma}\emph{\cite[V.A. Fact 1.10(4)]{sh300a}}\mylabel{fsstlem1}{Lemma \thetheorem} Using the above notation,
\begin{align*}
\card{\{p_\kappa\restriction \phi:p_\kappa\text{ does not $\phi$-split over }\Pi\}}&\leq2^{2^{|\Pi|}}\\
\card{\{p_\kappa:p_\kappa\text{ does not split over }\Pi\}}&\leq2^{2^{|\Pi|}\cdot{\chi}}
\end{align*} where $\chi\defeq|\hat{L}|^{<\kappa}=\big(|\gs^{<\kappa}(\emptyset)|\big)^{<\kappa}\leq 2^{<(\kappa+\ls^+)}$ is the size of the set of quantifier-free formulas of $\hat{L}_{\kappa,\kappa}$.
\end{lemma}
\begin{proof}
We count the number of combinations to build a $p_\kappa\restriction\phi$ that does not $\phi$-split over $\Pi$. Partition the parameters of $\phi$ by their $\Pi$-type. Namely, $b,c$ are equivalent iff for any $\hat{M}$ containing $b,c$ and the parameters from $\Pi$, any $\psi(y;d)\in\Pi$ with $l(y)=l(b)=l(c)$, we have $\hat{M}\vDash\psi[b;d]\Leftrightarrow \hat{M}\vDash\psi[c;d]$. Then there are $2^{|\Pi|}$-many classes. Within each class, say containing $b$, it remains to choose whether $\phi(x;b)$ or $\neg\phi(x;b)$ is in $p_\kappa\restriction\phi$. Hence we have $2^{2^{|\Pi|}}$-many choices.

The second part follows from the observation that a $(<\kappa)$-syntactic type $p_\kappa$ is determined by its restrictions $p_\kappa\restriction\phi$ where $\phi$ is a quantifier-free formula of $\hat{L}_{\kappa,\kappa}$.
\end{proof}
\begin{definition}\mylabel{fsstdef}{Definition \thetheorem}
A set $A$ is \emph{$(<\chi)$-compact} if for any $\hat{M}$ containing $A$, any cardinal $\lambda<\chi$, any quantifier-free formulas $\{\phi_i(x):i<\lambda\}$ in $\hat{L}_{\kappa,\kappa}$ with parameters from $A$, if $\hat{M}\vDash\bigwedge_{i<\lambda}\phi_i[b]$ for some $b\in\hat{M}$, then $\hat{M}\vDash\bigwedge_{i<\lambda}\phi_i[a]$ for some $a\in A$.
\end{definition}
\begin{lemma}\emph{\cite[V.A. Theorem 1.12]{sh300a}}\mylabel{fsstlem2}{Lemma \thetheorem}
Let $\chi$ be an infinite cardinal, $A$ be $(<\chi)$-compact with $A\subseteq|\hat{M}|$, $\phi(x;y)$ be a quantifier-free formula in $\hat{L}_{\kappa,\kappa}$. Either
\begin{enumerate}
\item For any $m\in|\hat{M}|$, there is a set $\Pi\subseteq\{\phi(x;a):a\in [A]^{<\kappa}\}$ such that $\card{\Pi}<\chi$ and $\tp_\kappa(m/A;\hat{M})\restriction\phi$ does not $\phi$-split over $\Pi$; or
\item $A\subseteq|\hat{M}|$ witnesses the $(<\kappa)$-syntactic order property of length $\chi$. 
\end{enumerate}
\end{lemma}
\begin{proof}
Suppose (1) does not hold, then we can pick $m\in|\hat{M}|$ such that $\tp_\kappa(m/A;\hat{M})\restriction\phi$ splits over any $\Pi$ with $|\Pi|<\chi$. Thus we can recursively build 
\begin{enumerate}
\item $\langle m_i,b_i,c_i:i<\chi \rangle$ inside $A$.
\item For $j<\chi$, $\hat{M}\vDash\phi[m;b_j]\leftrightarrow\neg\phi[m;c_j]$
\item For $i<j<\chi$, $\hat{M}\vDash\phi[m_i;b_j]\leftrightarrow\phi[m_i;c_j]$.
\item For $j<\chi$, $m_j\vDash\bigwedge_{i\leq j}\big(\phi(x;b_i)\leftrightarrow\neg\phi(x;c_i)$\big).
\end{enumerate}
The construction is possible by the definition of $\phi$-splitting and by $(<\chi)$-compactness of $A$. The sequence $\langle m_ib_ic_i:i<\chi\rangle$ witnesses the $(<\kappa)$-syntactic order property of length $\chi$ via the formula $\phi(x;y)\leftrightarrow\phi(x;z)$. 
\end{proof}
\begin{proof}[Proof of \ref{fsstthm}]
Let $A$ be a set of size $\mu$. As in \ref{fsstlem1}, $\chi$ bounds the number of quantifier-free formulas of $\hat{L}_{\kappa,\kappa}$. Also $\mu=\mu^{<\chi}$, so we may assume $A$ is $(<\chi)$-compact (see \ref{fsstdef}). Since \ref{fsstlem2}(2) fails, (1) must hold for each quantifier-free formula $\phi$ of $\hat{L}_{\kappa,\kappa}$. 

Now we count the number of $(<\kappa)$-syntactic types. Each type $p_\kappa$ is determined by its restrictions $\{p_\kappa\restriction\phi:\phi\text{ is a quantifier-free formula in }\hat{L}_{\kappa,\kappa}\}$. Since $\mu=\mu^{\chi}$, we may assume $p_\kappa=p_\kappa\restriction\phi$ for a fixed $\phi$ (this is where we need $\mu=\mu^\chi$ instead of $\mu=\mu^{<\chi}$). By  \ref{fsstlem2}(1), we can find some $\Pi_{p_\kappa}$ of size $<\chi$ such that $p_\kappa$ does not $\phi$-split over $\Pi_{p_\kappa}$. There are $[A]^{<\chi}=\mu^{<\chi}$-many ways to choose $\Pi_{p_\kappa}$. For each fixed $\Pi=\Pi_{p_\kappa}$, \ref{fsstlem1} gives at most $2^{2^{|\Pi|}}=2^{2^{<\chi}}$-many choices for $p_\kappa$. So in total there are $\mu^{<\chi}+2^{2^{<\chi}}=\mu$-many choices.
\end{proof}
\begin{corollary}\mylabel{fscor2}{Corollary \thetheorem} Let $K$ be a stable AEC.
\begin{enumerate}
\item \emph{\cite[Theorem 4.13]{s5}} If $K$ is $(<\kappa)$-tame, has $AP$ and is stable in some cardinal $\geq\kappa^-$, then the first stability cardinal is bounded above by $\beth_{(2^{<(\kappa+\ls^+)})^+}$.
\item
If $K$ is $(<\kappa)$-tame and does not have $(<\kappa)$-order property of length $\chi\defeq2^{<(\kappa+\ls^+)}$. then the first stability cardinal is bounded above by $2^{2^{<\chi}}$.
\item If $K$ is $\ls$-tame and does not have $\ls$-order property of length $2^{\ls}$, then the first stability cardinal is bounded above by $\beth_3(\ls)$.
\item Let $|D(T)|\defeq|\gs^{<\omega}(\emptyset)|$. If $K$ is $(<\al)$-tame and does not have $(<\omega)$-order property of length $|D(T)|$, then the first stability cardinal is bounded above by $\beth_2(|D(T)|)$.
\end{enumerate}
\end{corollary}
\begin{proof}
We prove (1): Since ${\bf K}$ is $(<\kappa)$-tame, by \ref{gmprop} $(<\kappa)$-syntactic stability in ${\bf\hat{K}}$ is equivalent to $(<\kappa)$-stability in ${\bf K}$. Also, by the contrapositive of \ref{fsprop3}(1)(a), no $(<\kappa)$-order property of length $\chi$ in ${\bf K}$ implies no $(<\kappa)$-syntactic order property in ${\bf\hat{K}}$ of length $\chi$.

Let ${\bf K}$ be stable in some $\lambda\geq\kappa^-$. Since it is $(<\kappa)$-tame, \ref{gvstab} gives stability in $\beth_{\lambda^+}(\lambda)$ and \ref{bonthm} gives $(<\kappa)$-stability in $\beth_{\lambda^+}(\lambda)$. By \ref{fsprop2}(3), there is $\chi<\beth_{(2^{<(\kappa+\ls^+)})^+}$ such that ${\bf K}$ does not have $(<\kappa)$-order property of length $\chi$. We may assume $\chi\geq2^{<(\kappa+\ls^+)}$. By the first paragraph and \ref{fsstthm}, $K$ is $(<\kappa)$-stable in $2^{2^{<\chi}}<\beth_{(2^{<(\kappa+\ls^+)})^+}$.
\end{proof}
\begin{remark}\mylabel{fsstrmk}{Remark \thetheorem}
In particular (4) misses the actual lower bound by $\beth_2$ \cite[III Theorem 5.15]{sh90}. 
\end{remark}
We now show the promised result in \ref{fsrmk2}:
\begin{corollary}\mylabel{fsprom}{Corollary \thetheorem}
Let $\lambda$ be an infinite cardinal and $\gamma$ be an ordinal with $\lambda\leq\gamma<(2^\lambda)^+$. Then there is a stable AEC ${\bf K}$ such that $\ls=\lambda$, ${\bf K}$ has the order property of length up to $\beth_{\gamma}(\lambda)$. Moreover, ${\bf K}$ has $JEP$, $NMM$, $(<\al)$-tameness but not $AP$.
\end{corollary}
\begin{proof}
Let $\chi\defeq\beth_{\gamma}(\lambda)$.
We use the example in \ref{fsprop} with $\alpha=\gamma+3$. 
Suppose ${\bf K}$ has no $(<\omega)$-order property of length $\chi$.
Since ${\bf K}$ is $(<\al)$-tame, by \ref{gmprop} $(<\omega)$-syntactic stability in ${\bf\hat{K}}$ is equivalent to $(<\omega)$-stability in ${\bf K}$. Also, by the contrapositive of \ref{fsprop3}(1)(a), no $(<\omega)$-order property of length $\chi$ in ${\bf K}$ implies no $(<\omega)$-syntactic order property in ${\bf\hat{K}}$ of length $\chi$. Since $\chi\geq 2^{<\al}$, by \ref{fsstthm}, ${\bf K}$ is $(<\omega)$-stable in all $\mu=\mu^\chi+2^{2^{<\chi}}$. In particular, it is $(<\omega)$-stable in $\beth_2({\chi})=\beth_{\alpha-1}(\lambda)=\beth_{\alpha-1}(\lambda)$, contradicting the fact that $M^*\in K$ is unstable in any cardinal $<\beth_{\alpha}(\lambda)$. 
\end{proof}
\begin{remark}
In our example, where exactly is the witness to the $(<\omega)$-order property of $\beth_{\gamma}(\lambda)$? Tracing the proofs, the key is the recursive construction in \ref{fsstlem2}, where a long splitting chain is utilized. Fix a cardinal $\chi\leq\beth_{\alpha-1}(\lambda)$. For our ${\bf K}$, we do not even need the $b_i$ and can simply set $\langle c_i,m_i:i<\chi\rangle$ such that all elements are distinct and $m_i$ contains exactly $\{c_j:j\leq i\}$. Then $\langle c_im_i:i<\chi\rangle$ witnesses the $2$-order property of length $\chi$, via the formula $\phi(x_1y_1;x_2y_2)\defeq x_1E y_2$. Therefore, we have an explicit example of the order property up to length $\beth_{\alpha-1}(\lambda)$ (the subscript cannot go further because most elements on the top rank do not belong to any other elements).
\end{remark}
We can conclude: 
\begin{corollary} 
\begin{enumerate}
\item For stable AECs, the Hanf number for the order property length is exactly $\beth_{(2^{\ls})^+}$.
\item For stable AECs, the Hanf number for the order property* length is exactly $\beth_{(2^{\ls})^+}$.
\item The Hanf number for stability is at least $\beth_{(2^{\ls})^+}$.  In other words, let $\lambda\geq\al$ and $\mu<\beth_{(2^\lambda)^+}$, there is a stable AEC ${\bf K}$ such that $\ls=\lambda$ and the first stability cardinal is greater than $\mu$.
\item With $\ls$-tameness and $AP$, the Hanf number for stability is at most $\beth_{(2^{\ls})^+}$. In other words, if ${\bf K}$ is a stable AEC with $AP$ and $\ls$-tameness, then the first stability cardinal is at most $\han$.
\end{enumerate}
\end{corollary}
\begin{proof}
\begin{enumerate}
\item Lower bound is by \ref{fsprom} and upper bound is by \ref{fsprop2}(3).
\item Lower bound is by \ref{fsprop} and upper bound is by \ref{fsprop2}(3).
\item By \ref{fsprop}.
\item By \ref{fscor2}(1).
\end{enumerate}
\end{proof}
We finish this section with the following question: are the bounds in (3) and (4) optimal?

\section{Syntactic splitting}\label{synspl}
We will give a syntactic proof to \ref{gvstab} using Galois Morleyization. The advantage is that types are syntactic and can be over sets of size less than $\ls$; the disadvantage is that we have extra assumptions. 
\begin{assum}
We assume the existence of a monster model $\mn$ ($AP+JEP+NMM$), where each set is inside some model in $K$ and each (set) embedding/isomorphism is extended by a ${\bf K}$-embedding/isomorphism. We also assume $AP$ over set bases: if $A\subseteq|M_1|\cap|M_2|$ and $M_1,M_2$ interpret $A$ in the same way, then there are $M_3\geq M_2$ and $f:M_1\xrightarrow[A]{}M_3$.
\end{assum}

\begin{definition}Let $\mu$ be an infinite cardinal, $A,B$ be sets in some models of $K$.
\begin{enumerate}
\item $B$ is \emph{universal} over $A$ if $A\subset B$ and for any $|B'|\leq|B|$, $B'\supseteq A$, there is $f:B'\xrightarrow[A]{}B$. We write $A\subset_u B$.
\item $B$ is \emph{$\mu$-universal} over $A$ if $B'$ in (1) must have size $\leq\mu$.
\item $B$ is \emph{$\mu$-homogeneous} if it is $<\mu$-universal over any $C\subset B$ of size $<\mu$.
\item Let $A\subseteq B\in K$, $p\in\gs(B)$. We say $p$ \emph{$\mu$-splits over $A$} if there exists $A\subseteq B_1,B_2\subseteq B$, $\nr{B_1}=\mu$, $f:B_1\cong_AB_2$ with $f(p)\restriction B_2\neq p\restriction B_2$. 
\end{enumerate}
\end{definition}
\cite[II 1.16(1)(a)]{shh} shows that universal \emph{models} of size $\mu$ exist if ${\bf K}$ is stable in $\mu\geq\ls$ (and has $\mu$-$AP$, $\mu$-$NMM$). 
We prove universal sets exist:
\begin{proposition}\mylabel{gvstabprop}{Proposition \thetheorem}
\begin{enumerate}
\item For any $A$, if ${\bf K}$ is stable in $|A|$, there is $|B|=|A|$, $B\supset_u A$. 
\item For any infinite $\lambda,\mu$ and any $|A|\leq\lambda$. If $\lambda^{<\mu}=\lambda$, then there is a $\mu$-homogeneous $B\supset A$ of size $\lambda$. 
\end{enumerate}
\end{proposition}
\begin{proof}
\begin{enumerate}
\item Let $\mu\defeq|A|$. Build $\langle B_i:i<\mu\rangle$ increasing and continuous such that $B_0\defeq A$, $B_{i+1}\vDash\gs(B_i)$. For any $A'\supseteq A$, $|A'|\leq\mu$. We may assume $|A'-A|=\mu$ and write $A'=A\cup\{a_i:i<\mu\}$. Define $A_i\defeq A\cup\{a_j:j\leq i\}$ and $\langle f_i:A_i\xrightarrow[A]{}B_i:i\leq\mu\rangle$ increasing and continuous partial embeddings such that $f_i(a_i)\in B_i$. Set $f_{-1}\defeq\oop{id}_A$ and suppose $f_i$ has been constructed, obtain $\tilde{A'}$, an isomorphic copy of $A'$ over $\ran (f_i)$ and denote $\tilde{a}_{i+1}$ the copy of $a_{i+1}$. Now $B_{i+1}\vDash\gs(B_i)\supseteq\gs(\ran(f_i))$ so it realizes the type of $\tilde{a}_{i+1}$ over $\ran (f_i)$, say by $b_{i+1}$. By $AP$ there is $g:\tilde{a}_{i+1}\mapsto b_{i+1}$ fixing $\ran(f_i)$. Define $f_{i+1}(a_{i+1})\defeq b_{i+1}$. 
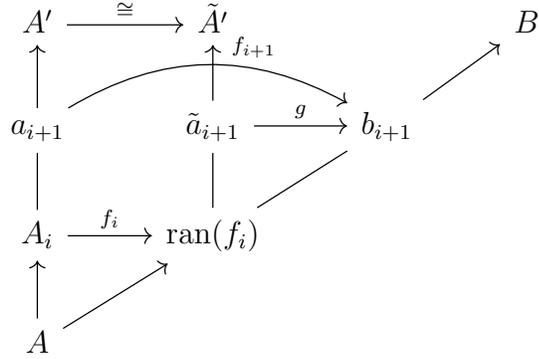
\begin{center}
\begin{tikzcd}
A' \arrow[r, "\cong"]                              & \tilde{A'}                                                     &                    & B \\
a_{i+1} \arrow[u] \arrow[rr, "\quad\qquad f_{i+1}", bend left] & \tilde{a}_{i+1} \arrow[u] \arrow[r, "g"]             & b_{i+1} \arrow[ru] &   \\
A_i \arrow[u, no head] \arrow[r, "f_i"]            & \operatorname{ran}(f_i) \arrow[ru, no head] \arrow[u, no head] &                    &   \\
A \arrow[u] \arrow[ru]                             &                                                                &                    &  
\end{tikzcd}
\end{center}
\item By an exhaustive argument, we can build a (set) saturated $B\supset A$. We check that it is $\mu$-homogeneous. Let $C\subset B$, $C'\supseteq C$ both of size $<\mu$, the argument from the previous item applies because $B$ is saturated and $|C'|<\mu$. 
\end{enumerate}
\end{proof}

We notice a correspondence between $\mu$-Galois splitting and $(<\mu^+)$-syntactic splitting (see \ref{syndef}); a similar treatment for coheir has already been done in \cite[Section 5]{s5}). We write $q$ \emph{$\mu$-syn-splits over $A$} to mean $q$ $(<\mu^+)$-syntactically splits over the quantifier-free formulas of $\hat{L}_{\mu^+,\mu^+}$ over $A$.
\begin{proposition}\mylabel{gsynspl}{Proposition \thetheorem}
Let $\mu$ be an infinite cardinal, ${\bf \hat{K}}$ be the $(<\mu^+)$-Galois Morleyization of ${\bf K}$. For any $A\subseteq B\in K$, $p\in\gs(B)$, $p$ $\mu$-splits over $A$ iff $p_{\mu^+}$ $\mu$-syn-splits over $A$.
\end{proposition}
\begin{proof}
Suppose $d\vDash p$ $\mu$-splits over $A$, obtain witness $f:B_1\cong_AB_2$ as above. Enumerate $A=a$ and $B_i$ as $b_i$. Since $f(p)\restriction B_2\neq p\restriction B_2$, $r_1\defeq\gtp(f(d)b_2/\emptyset)\neq\gtp(db_2/\emptyset)$, so $\hat{\mn}\vDash R_{r_1}[db_1]\wedge R_{r_1}[f(d)b_2]\wedge\neg R_{r_1}[db_2]$. $b_1$ and $b_2$ have the same syntactic type over $a$ because of $f$. Therefore, $p_{\mu^+}=\tp_{\mu^+}(d/B)\supseteq\tp_{\mu^+}(d/B_1\cup B_2)$ $\mu$-syn-splits over $A$.

Conversely, suppose $p_{\mu^+}$ $\mu$-syn-splits over $A$. There are $\phi(x;b_1),\neg\phi(x;b_2)$ in $p_{\mu^+}$ such that $b_1,b_2$ have the same syntactic type over $a$. Pick $d\vDash p_{\mu^+}$, then $\hat{\mn}\vDash\phi[d;b_1]\wedge\neg\phi[d;b_2]$, $\gtp(db_1/\emptyset)\neq\gtp(db_2/\emptyset)$ (actually $\phi$ might tell the exact Galois type of $db_1$). On the other hand, let $r=\gtp(b_1a/\emptyset)$ and consider $R_r(x;a)$. As $b_1,b_2$ have the same syntactic type over $A$, $\hat{\mn}\vDash R_r[b_1;a]\wedge R_r[b_2;a]$, which means $\gtp(b_2a/\emptyset)=r$. Thus there is $f:b_1\cong_ab_2$.
\end{proof}
In the above proof, we did not use tameness simply because $(<\mu^+)$-Galois Morleyization is already large enough to code all types over domains of size $\mu$. 
\begin{corollary}\mylabel{gsynsplcor}{Corollary \thetheorem}
Let $\mu$ be an infinite cardinal and ${\bf K}$ be $\mu$-tame. Let $A\subseteq B$ with $|A|\leq \mu$. Then any $p\in S(B)$ $(\geq\mu)$-splits over $A$ iff it $\mu$-splits over $A$.
\end{corollary}
\begin{proof}
We adopt the previous proof: suppose $d\vDash p$ $(\geq\mu)$-splits over $A$, obtain witness $f:B_1\cong_AB_2$ as above. Enumerate $A=a$ and $B_i$ as $b_i$. Since $f(p)\restriction B_2\neq p\restriction B_2$, $\gtp(f(d)b_2/\emptyset)\neq\gtp(db_2/\emptyset)$. By \ref{gmprop}, there is a quantifier-free formula $\phi$ in $\hat{L}_{\mu^+,\mu^+}$ so that $\hat{\mn}\vDash \phi[db_1]\wedge \phi[f(d)b_2]\wedge\neg \phi[db_2]$. $b_1$ and $b_2$ have the same syntactic type over $a$ because of $f$. Therefore, $p_{\mu^+}=\tp_{\mu^+}(d/B)\supseteq\tp_{\mu^+}(d/B_1\cup B_2)$ $\mu$-syn-splits over $A$.
This implies $p$ $\mu$-splits over $A$ by the second paragraph of the previous proof.
\end{proof}
We now prove a series of results syntactically. The original proofs in \cite[Theorems I 4.10,4.12]{van06}, \cite[Section 6]{GV06b} are semantic. 
\begin{lemma}\mylabel{gvstablem0}{Lemma \thetheorem}
Let $\mu$ be an infinite cardinal and ${\bf K}$ be $\mu$-tame. For any $B\subseteq C$ both of size $\geq\mu$ and $p\in\gs(C)$, we have $p_{\mu^+}\restriction B=(p\restriction B)_{\mu^+}$.
\end{lemma}
\begin{proof}Let $d\vDash p$,
\begin{align*}
p_{\mu^+}\restriction B&=\tp_{\mu^+}(d/C)\restriction B\\
&=\tp_{\mu^+}(d/B)\\
&=(p\restriction B)_{\mu^+}\text{ because $d\vDash p\restriction B$}
\end{align*}
\end{proof}
\begin{lemma}\mylabel{gvstablem1}{Lemma \thetheorem}
Let $\mu$ be an infinite cardinal, $A\subset_u B\subseteq C$ all of size $\mu$ and $p,q\in \gs(C)$. Suppose $p,q$ do not $\mu$-split over $A$ and $p\restriction B=q\restriction B$. Then $p=q$. We can allow $|B|,|C|\geq\mu$ if we assume $\mu$-tameness.
\end{lemma}
\begin{proof}
We comment the case $|C|>\mu$ in square brackets. Let ${\bf \hat{K}}$ be the $(<\mu^+)$-Galois Morleyization of ${\bf K}$. Since $p,q$ do not $\mu$-split over $A$, \ref{gsynspl} shows that $p_{\mu^+},q_{\mu^+}$ do not $\mu$-syn-split over $A$ [use $\mu$-tameness and \ref{gsynsplcor}]. Suppose $p_{\mu^+}\neq q_{\mu^+}$, [by $\mu$-tameness] there is $d\subseteq C$ of size $\mu$, $\phi(x;y)$ such that $\phi(x;d)\in p_{\mu^+}-q_{\mu^+}$. As $B\supset_u A$, we may pick $b\vDash\tp_{\mu^+}(d/A)$. By non-syn-splitting, $\phi(x;b)\in p_{\mu^+}-q_{\mu^+}$ contradicting $p_{\mu^+}\restriction B=q_{\mu^+}\restriction B$ [use \ref{gmprop} and \ref{gvstablem0}].
\end{proof}

Extension also holds but it is applicable to $\mu$-sized models.
\begin{lemma}\mylabel{gvstablem2}{Lemma \thetheorem}
Let $\mu$ be an infinite cardinal, $A\subset_u B\subseteq C$ all of size $\mu$. Let $p\in\gs(B)$ do not $\mu$-split over $A$. Then there is $q\in\gs(C)$ extending $p$ and does not $\mu$-split over $A$. Also, if $p$ is non-algebraic, we can have $q$ non-algebraic. 
\end{lemma}
\begin{proof}
Let ${\bf \hat{K}}$ be the $(<\mu^+)$-Galois Morleyization of ${\bf K}$. First we decide whether $\pm\phi(x;c)\in q_{\mu^+}$ for each quantifier-free formula in $\hat{L}_{\mu^+,\mu^+}$ over $C$. Since $A\subset_u B$, there is a copy $C'\subset B$ of $C$. We want $q$ does not $\mu$-split over $A$, by \ref{gsynspl} we must set $\phi(x;c)\in q_{\mu^+}$ iff $\phi(x;b_c)\in p_{\mu^+}$ where $b_c\in C'$ and $b_c\vDash\tp_{\mu^+}(c/A)$. Such $q_{\mu^+}$ is realized because $p_{\mu^+}$, and thus $p_{\mu^+}\restriction C'\cong_A q_{\mu^+}$ is realized where $C'\subset B$ is the copy of $C$. 

If $p$ is non-algebraic and we can modify the argument above by extending $C$ to a copy of $B$. Then $p_{\mu^+}$ is realized/algebraic iff $q_{\mu^+}$ is.
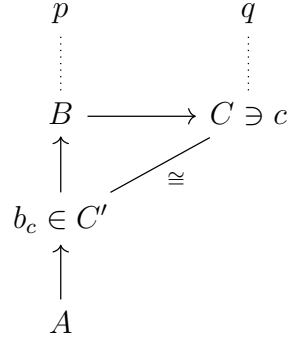
\begin{center}
\begin{tikzcd}
p&q\\
B \arrow[r]\arrow[u, no head, dotted]  & C\ni c \arrow[ld, "\cong", no head]\arrow[u, no head, dotted] \\
b_c\in C' \arrow[u] &                                \\
A \arrow[u]  &                               
\end{tikzcd}
\end{center}
\end{proof}

\begin{lemma}\mylabel{gvstablem3}{Lemma \thetheorem}
Let $\mu$ be an infinite cardinal, ${\bf K}$ be $\mu$-tame and stable in $\mu$. For any $|A|\leq\mu$, $A\subset_u C$, $$\chi\defeq\card{\{p\in\gs(C):p\text{ does not $\geq\mu$-split over $A$}\}}\leq\mu$$
\end{lemma}
\begin{proof}
Pick $B$ of size $\mu$ with $A\subset_u B\subseteq C$. By \ref{gsynsplcor}, \ref{gvstablem1} and \ref{gvstablem2}, $\chi=\card{\{p\in\gs(B):p\text{ does not $\mu$-split over $A$}\}}\leq|\gs(B)|\leq\mu$.
\end{proof}
The following originates from \cite[Claim 3.3]{sh394} and is extended to longer types in \cite[Fact 4.6]{GV06b}.
\begin{lemma}\mylabel{gvstablem4}{Lemma \thetheorem}
Let $\mu$ be an infinite cardinal and ${\bf K}$ be stable in $\mu$. For any $p\in\gs(B)$, there is $A\subseteq B$ of size $\mu$ such that $p$ does not $\mu$-splits over $A$.
\end{lemma}
\begin{proof}
Suppose the lemma is false, we can find $d\vDash p\in\gs(B)$, $B$ of size $>\mu$ such that $p$ $\mu$-splits over all $A$ of size $\mu$. Let ${\bf \hat{K}}$ be the $(<\mu^+)$-Galois Morleyization of ${\bf K}$. By \ref{gsynspl}, $p_{\mu^+}$ $\mu$-syn-splits over all $A$ of size $\mu$. Pick any $A_0\subset B$ of size $\mu$ and choose minimum $\kappa\leq\mu$ with $2^\kappa>\mu$. By assumption we can build $\langle a_{\eta^\frown 0},a_{\eta^\frown 1},\phi_\eta,f_\eta:\eta\in2^{<\kappa}-\{\langle\rangle\}\rangle$ and $\langle A_\alpha:\alpha<\kappa\rangle$ such that 
\begin{enumerate}
\item $\langle A_\alpha:\alpha<\kappa\rangle$ is increasing and continuous. For $\alpha<\kappa$, $A_\alpha\subset B$ has size $\mu$.
\item For $\eta\in2^{<\kappa}$,  $a_\eta\in A_{l(\eta)}$, $f_\eta:a_{\eta^\frown0}\mapsto a_{\eta^\frown1}$ and $f_\eta$ fixes $A_{l(\eta)}$.
\item For $\eta\in2^{<\kappa}$, $\phi_\eta(x;a_{\eta^\frown 0}),\neg\phi_\eta(x;a_{\eta\frown1})\in p_{\mu^+}$
\end{enumerate}
\begin{center}
\begin{tikzcd}
A_2 & \phi_0(x;a_{00}) \arrow[rr, "f_0","A_1"'] &                                       & \phi_0(x;a_{01}) & \phi_1(x;a_{10}) \arrow[rr, "f_1","A_1"'] &     & \phi_1(x;a_{11}) \\
A_1 &                                    & a_0 \arrow[rrr, "f_{\langle\rangle}","A_0"'] &                  &                                    & a_1 &             \\     
\end{tikzcd}
\end{center}

For $\nu\in 2^\kappa$, define $q_\nu\defeq\{\phi_\eta(x;a_{\eta^\frown i}):\eta\in2^{<\kappa},\eta^\frown i\sqsubseteq \nu;i=1,2\}\cup\{\neg\phi_\eta(x;a_{\eta^\frown1}):\eta\in2^{<\kappa},\eta^\frown0\sqsubseteq\nu\}$. $\langle q_\nu:\nu\in 2^\kappa\rangle$ are obviously pairwise contradictory. It remains to show each of them is realized. For any $\nu\in2^\kappa$, define $g_\nu$ to be the composition of $\langle f_{\nu\restriction\alpha}^{\nu[\alpha]}:\alpha<\kappa\rangle$ where each $f_\eta^0\defeq\oop{id}$ and $f_\eta^1\defeq f_\eta$. $g_\nu$ is well-defined because for $\eta\in2^{<\kappa}$, $f_\eta$ fixes $A_{l(\eta)}$. Extend $g_\nu$ to an isomorphism containing $d$. Inductively we can show $g_\nu(d)\vDash q_\nu$: let $\alpha<\kappa$. If $\nu[\alpha]=0$, then item (3) in the construction guarantees $g(p_\mu^+)\supseteq g_\nu(\{\phi_{\nu\restriction\alpha}(x;a_{{\nu\restriction\alpha}^\frown 0}),\neg\phi_{\nu\restriction\alpha}(x;a_{{\nu\restriction\alpha}\frown1})\})=\{\phi_{\nu\restriction\alpha}(x;a_{{\nu\restriction\alpha}^\frown 0}),\neg\phi_{\nu\restriction\alpha}(x;a_{{\nu\restriction\alpha}\frown1}\})$. If $\nu[\alpha]=1$, then $g(p_\mu^+)\supseteq g_\nu(\{\phi_{\nu\restriction\alpha}(x;a_{{\nu\restriction\alpha}^\frown 0})\})=\{\phi_{\nu\restriction\alpha}(x;a_{{\nu\restriction\alpha}^\frown 1})\}$.
\end{proof}
\begin{proof}[Proof of \ref{gvstab}]
Let $C$ be of size $\lambda$. By an exhaustive argument, we may extend $C$ to be $\mu^+$-saturated. Let $A\subset C$ of size $\mu$. By \ref{gvstabprop}(1) and (2), we can build $B$ of size $\mu$ such that $A\subset_uB\subset C$. Thus \ref{gvstablem3} applies. Also by \ref{gvstablem4}, each $p\in \gs(C)$ does not $\mu$-split over some $A_p\subset C$ of size $\mu$. There are at most $\lambda^\mu$ of such $A_p$. Thus$$
\card{\gs(C)}=\big|\bigcup\{p\in\gs(C):p\text{ does not $\mu$-split over $A_p$}\}\big|
\leq\lambda^\mu\cdot\mu=\lambda$$
\end{proof}
\begin{corollary}
If ${\bf K}$ is stable in some $\lambda<\ls$. Then the first stability cardinal $\geq\ls$ is bounded above by $2^{\ls}$. 
\end{corollary}
\begin{proof}
Apply \ref{gvstab} to $(2^{\ls})^\lambda=2^{\ls}$.
\end{proof}

Our final application is the upward transfer of stability. The original proof of (1) below  uses weak tameness (tameness over saturated models). \cite[Lemma 5.5]{s3} proves a stronger version of (2) with chain local character instead of set local character, but we do not assume the former here. 
\begin{proposition}
Let $\mu<\lambda$ be infinite cardinals. Assume ${\bf K}$ is $\mu$-tame and stable in $\mu$.
\begin{enumerate}
\item \emph{\cite[Theorem 4.5]{bkv}} ${\bf K}$ is also stable in $\mu^+$. 
\item  If in addition $\cf(\lambda)>\mu$ and ${\bf K}$ is stable in unbounded many cardinals below $\lambda$, then it is stable in $\lambda$.
\end{enumerate} 
\end{proposition}
\begin{proof}
\begin{enumerate}
\item Suppose $|A|=\mu^+$ has $\mu^{++}$ many types $\langle p_\alpha:\alpha<\mu^{++}\rangle$. Write $A=\bigcup_{i<\mu^+}A_i$ increasing and continuous, with $|A_i|=\mu$. We may assume $A_{i+1}\vDash\gs(A_i)$ by the following: define another chain $\langle A_i':i\leq\mu^+\rangle$ increasing and continuous such that $\card{A_{i+1}'}=\card{A_{i+1}}=\mu$ and $A_{i+1}'\supset_u A_i\cup A_i'$ (using \ref{gvstabprop}(1)). Replace $A$ by $A_{\mu^+}'$. 

By \ref{gvstablem4} and $\cf{\mu^{++}}>\mu^+$, we may assume all $p_\alpha$ does not $\mu$-split over $A_0$. By stability in $\mu$ and pigeonhole principle, we may assume all $p_\alpha$ has the same type over $A_1$. Together with \ref{gvstablem1}, all $p_\alpha$ are equal, contradiction.

\item We consider limit cardinal $\lambda$. Pick a cofinal sequence $\langle \lambda_i:i<\cf(\lambda)\rangle$ to $\lambda$ such that ${\bf K}$ is stable in all $\lambda$. Repeat the same argument as (1) but with $|A|=\lambda$ and $A=\bigcup_{i<\cf(\lambda)}A_i$. 
\end{enumerate}
\end{proof}

\bibliographystyle{alpha}
\bibliography{references}

{\small\setlength{\parindent}{0pt}
\textit{Email}: wangchil@andrew.cmu.edu

\textit{URL}: http://www.math.cmu.edu/$\sim$wangchil/

\textit{Address}: {Department of Mathematical Sciences, Carnegie Mellon University, Pittsburgh PA 15213, USA}
\end{document}